\documentclass[a4paper,usenames,dvipsnames]{amsart}

\usepackage{amsmath} 
\usepackage{amstext}
\usepackage{amsthm}
\usepackage{amscd} 
\usepackage{amsopn} 
\usepackage{verbatim} 
\usepackage{amssymb}
\usepackage{amsfonts}
\usepackage{mathtools}
\usepackage[mathscr]{euscript}
\usepackage{fullpage}
\usepackage[bbgreekl]{mathbbol}
\usepackage{tikz}
\usepackage{tikz-cd}
\usetikzlibrary{matrix}
\usetikzlibrary{shapes}
\usetikzlibrary{arrows}
\usetikzlibrary{calc,3d}
\usetikzlibrary{decorations,decorations.pathmorphing}
\usetikzlibrary{through}
\tikzset{ext/.style={circle, draw,inner sep=1pt},int/.style={circle,draw,fill,inner sep=1pt},nil/.style={inner sep=1pt}}
\tikzset{exte/.style={circle, draw,inner sep=3pt},inte/.style={circle,draw,fill,inner sep=3pt}}
\tikzset{diagram/.style={matrix of math nodes, row sep=3em, column sep=2.5em, text height=1.5ex, text depth=0.25ex}}
\tikzset{diagram2/.style={matrix of math nodes, row sep=0.5em, column sep=0.5em, text height=1.5ex, text depth=0.25ex}}
\tikzset{every picture/.append style={baseline=-.65ex}}
\usepackage{todonotes}

\usepackage{hyperref}

\usepackage{ulem}
\usepackage{color}
\newcommand{\udot}{{\:\raisebox{4pt}{\selectfont\text{\circle*{1.5}}}}}

\newcommand{\ttt}{\text{-}}

\let\leq\leqslant
\let\geq\geqslant

\newcommand{\NN}{\mathbb N}
\newcommand{\ZZ}{\mathbb Z}

\newcommand{\kk}{\Bbbk}
\newcommand{\uk}{\underline{\kk}}

\newcommand{\frg}{\mathfrak{g}}

\newcommand{\bbS}{\mathbb{S}}

\newcommand{\irchi}[2]{\raisebox{\depth}{$#1\chi$}} 
\DeclareRobustCommand{\rchi}{{\mathpalette\irchi\relax}}

\newcommand{\calB}{\mathcal{B}}

\newcommand{\calO}{\mathcal{O}}
\newcommand{\calP}{\mathcal{P}}
\newcommand{\dcalP}{\partial{\mathcal{P}}}
\newcommand{\calQ}{{\mathcal{Q}}}
\newcommand{\calF}{{\mathcal{F}}}

\newcommand{\calM}{\mathcal M}

\newcommand{\calV}{\mathcal V}
\newcommand{\calR}{\mathcal R}

\newcommand\Lie{\mathsf{Lie}}
\newcommand\PreLie{\mathsf{PreLie}}
\newcommand{\Ass}{\mathsf{As}}
\newcommand{\Pois}{{\mathsf{Pois}}}
\newcommand\Com{\mathsf{Comm}}

\newcommand{\Perm}{\mathsf{Perm}}
\newcommand\Leib{\mathsf{Leib}}
\newcommand\Zinb{\mathsf{Zinb}}

\newcommand{\antish}{\text{\raisebox{\depth}{\textexclamdown}}} %antishriek
\newcommand{\dual}{\vee}
\newcommand{\op}{\mathsf{op}}
\newcommand{\oprd}{\mathsf{Op}}
\newcommand{\alg}{\mathsf{Alg}}
\newcommand{\Sgn}{\mathsf{Sgn}}

\newcommand{\shfl}{\mathsf{Sh}}
\newcommand{\gr}{\mathsf{gr}}

\newcommand{\UU}{\mathsf{U}}
\newcommand{\UP}{\mathsf{U}_{\calP}}
\newcommand{\UPZ}{\mathsf{U}_{\calP}^{0}}

\newcommand{\perm}{\mathsf{perm}}

\newcommand{\dcirc}{
	{\begin{tikzpicture}[scale=0.1]
		\draw[thick,dash dot] (0,.2) circle (1cm);
		\end{tikzpicture}}
}

\newcommand{\binar}[1]{
\begin{tikzpicture}[scale=0.4]
\node[ext] (v) at (0,0) {\tiny{$#1$}};
\node (w0) at (0,-1) {};
\node (w1) at (-1,1) {\tiny{1}};
\node (w2) at (1,1) {\tiny{2}};
\draw (v) edge (w0) edge (w1) edge (w2);
\end{tikzpicture}
}	

\newcommand{\corolla}[1]{
\begin{tikzpicture}[scale=0.5]
\node[ext] (v0) at (0,0) {\tiny$#1$};
\node (w1) at (-1,1) {\tiny{1}};
\node (w2) at (1,1) {\tiny{2}};
\coordinate (r) at (0,-0.9);
\draw (v0) edge (w1) edge (r);
\draw (v0) edge (w2); 
\end{tikzpicture}
}

\newcommand{\lltree}[2]{
	\begin{tikzpicture}[scale=0.38]
	\node[ext] (v1) at (0,-1) {\tiny{$#1$}};
	\node[ext] (v2) at (-1,1) {\tiny{$#2$}};
	\coordinate (r) at (0,-2);
	\node (w1) at (-1.5,2.5) {${\tiny{1}}$};
	\node (w2) at (-0.5,2.5) {${\tiny{2}}$};
	\node (w3) at (0.5,1) {${\tiny{3}}$};
	\draw (v1) edge (r);
	\draw (v1) edge (v2) edge (w3);
	\draw (v2) edge (w1) edge (w2);
	\end{tikzpicture}} 
\newcommand{\lrtree}[2]{
	\begin{tikzpicture}[scale=0.38]
	\node[ext] (v1) at (0,-1) {\tiny{$#1$}};
	\node[ext] (v2) at (-1,1) {\tiny{$#2$}};
	\coordinate (r) at (0,-2);
	\node (w1) at (-1.5,2.5) {${\tiny{1}}$};
	\node (w2) at (-0.5,2.5) {${\tiny{3}}$};
	\node (w3) at (0.5,1) {${\tiny{2}}$};
	\draw (v1) edge (r);
	\draw (v1) edge (v2) edge (w3);
	\draw (v2) edge (w1) edge (w2);
	\end{tikzpicture}} 
\newcommand{\rrtree}[2]{
\begin{tikzpicture}[scale=0.38]
\node[ext] (v1) at (0,-1) {\tiny{$#1$}};
\node[ext] (v2) at (1,1) {\tiny{$#2$}};
\coordinate (r) at (0,-2);
\node (w1) at (1.5,2.5) {${\tiny{3}}$};
\node (w2) at (0.5,2.5) {${\tiny{2}}$};
\node (w3) at (-0.5,1) {${\tiny{1}}$};
\draw (v1) edge (r);
\draw (v1) edge (v2) edge (w3);
\draw (v2) edge (w1) edge (w2);
\end{tikzpicture}	
}

\numberwithin{equation}{section}

\newtheorem{theorem}[equation]{Theorem}
\newtheorem*{theorem*}{Theorem}
\newtheorem{proposition}[equation]{Proposition}
\newtheorem*{proposition*}{Proposition}

\newtheorem*{statement*}{Statement}
\newtheorem{lemma}[equation]{Lemma}
\newtheorem*{lemma*}{Lemma}
\newtheorem{corollary}[equation]{Corollary}
\newtheorem*{corollary*}{Corollary}

\newtheorem{definition}[equation]{Definition}
\newtheorem*{definition*}{Definition}
\theoremstyle{definition}

\newtheorem{remark}[equation]{Remark}
\newtheorem*{remark*}{Remark}
\newtheorem{example}[equation]{Example}
\newtheorem*{example*}{Example}

\title{PBW property for associative universal enveloping algebras over an operad}

\author{Anton Khoroshkin}
\address{	%	International Laboratory of Representation Theory and Mathematical Physics, 
	National Research University Higher School of Economics, 
	20 Myasnitskaya street, Moscow 101000, Russia  \& 	
	Institute for Theoretical and Experimental Physics, Moscow 117259, Russia; 	
}
\email{akhoroshkin@hse.ru}

\begin{document}

\maketitle	

\begin{abstract}	
	 Given a symmetric operad $\mathcal{P}$ and a $\mathcal{P}$-algebra $V$, the associative universal enveloping algebra ${\mathsf{U}_{\mathcal{P}}}$ is an associative algebra whose category of modules is isomorphic to the abelian category of $V$-modules. We study the notion of PBW property for universal enveloping algebras over an operad.
	 In case $\mathcal{P}$ is Koszul a criterion for the PBW property is found.  A necessary condition on the Hilbert series for $\mathcal{P}$ is discovered. Moreover, given any symmetric operad $\mathcal{P}$, together with a Gr\"obner basis $G$, a condition is given in terms of  the structure of the underlying trees associated with leading monomials of $G$, sufficient for the PBW property to hold.
	 Examples are provided.
\end{abstract}

\setcounter{section}{-1}
\section{Introduction}
Let $\calP$ be a symmetric operad and let $V$ be a $\calP$-algebra. 
The natural notion of a module $M$ over a $\calP$-algebra $V$ was given in the beginning of the epoch of operads (see e.g.~\cite{Ginz_Kapranov} \S1.6). At the same time a definition of a universal enveloping algebra $\UU_{\calP}(V)$ in terms of operadic trees was presented (e.g.~\cite{Ginz_Kapranov}  Definition 1.6.4,\cite{Hinich_Schechtman}). The main advantage is that the associative algebra $\UU_{\calP}(V)$ satisfies the following universal property: 
\begin{center}
{\it The abelian category of $V$-modules is equivalent to the category of left modules over $\UU_{\calP}(V)$.}
\end{center}
For example, one can easily verify that for a Lie algebra $\frg$ (viewed as a $\calP$-algebra for the symmetric operad $\calP=\Lie$) the operadic universal enveloping algebra $\UU_{\Lie}(\frg)$ coincides with the ordinary universal enveloping algebra $U(g)$ (see also Example~\ref{ex::Lie} below). 
Every textbook on Lie algebras includes the following statement known as the Poincar\'e-Birkhoff-Witt Theorem:
\begin{center}
{\it There exists a (functorial in $\frg$) filtration of the universal enveloping algebra $U(\frg)$ \\ such that  the associated graded algebra is isomorphic to the symmetric algebra $S(\frg)$. \\
In particular, there exists a (functorial in $\frg$) isomorphism of vector spaces $U(\frg) {\simeq} S(\frg)$. }
\end{center}
Let $\frg_0$ be the trivial Lie algebra structure on a given vector space $\frg$ (i.e. all commutators on $\frg_0$ are set to be zero).
The symmetric algebra $S(\frg)$ coincides with the universal enveloping algebra $U(\frg_0)$ and the PBW property can be interpreted as stating that the associated graded to $U(\frg)$ is isomorphic to $U(\frg_0)$.	
Restricting to Koszul operads $\calP$, we shall establish a criterion to select those operads whose universal enveloping functor
\[
\UP: \calP\text{-algebras} \rightarrow \text{ Associative algebras}
\]
satisfies the PBW property. 
In other words, we are interested in necessary and sufficient conditions on the Koszul operad $\calP$ that imply a (functorial in $V$) filtration of $\UP(V)$ and a functorial isomorphism between $\gr\UP(V)$ and $\UP(V_0)$. Here $V$ is a (nontrivial) $\calP$-algebra and $V_0$ is the corresponding trivial $\calP$-algebra: all nontrivial elements of $\calP$ act trivially (by zero) on $V_0$.
Before stating our main theorem we suggest rephrasing the notion of the universal enveloping functor in terms of colored operads.
 With each symmetric operad $\calP$ we assign a two-colored operad $\calP_{+}$ 
 by coloring the first input and the output in a second color
 (Definition~\ref{dfn::P+} below). The first color in $\calP_{+}$ corresponds to a $\calP$-algebra and the second color corresponds to a module over this algebra. 
 %In particular, the operations $\calP_{+}^1$ with the ouput of the first color assemble the operad $\calP$ and the operations $\calP_{+}^2$ with the output of the second color has only one input of the second color. The underlying algebraic structure given by composition on  $\calP_{+}^2$ is known under the name  \emph{twisted associative algebra}(see e.g.~\cite{Dotsenko_Bremner::Shuffle}) and the notation $\partial P$ stands for $\calP_{+}^2 :=Res_{\bbS_n}^{\bbS_{n+1}}\calP(n+1)$ in the combinatorial theory of species (\cite{Species}).
%
The colored suboperad of $\calP_{+}$ spanned by elements $\calP_{+}(n,1)^{2}$ with exactly one input and with the output of the second color consists of a collection of $\bbS_n$-representations $\dcalP(n):=Res_{\bbS_n}^{\bbS_{n+1}}\calP(n+1)$ that admit a right $\calP$-module structure.
Our first minor finding  states that the universal enveloping functor can be defined in the following way: $$\UP(V) = \dcalP\circ_{\calP} V.$$
 In particular, $\UP$  satisfies the PBW property if and only if  $\dcalP$ is  a free right $\calP$-module (Theorem~\ref{thm::PBW::dP::free}).

Note that the operadic composition through the second color on $\dcalP$ coincides with a well known monoidal product for symmetric collections (denoted by $\diamond$):
\[
(\dcalP\diamond\dcalP)(n):= \bigoplus_{k+m=n} Ind_{{\bbS}_{k}\times {\bbS_m}}^{\bbS_{n}} \dcalP(k)\otimes \dcalP(m).
\]
%The corresponding product of characters is given by multiplication of symmetric functions:
%\[\rchi_{\bbS_{m+n}}(\dcalP(m)\diamond \dcalP(n)) = \rchi_{\bbS_{m}}(\dcalP(m))\cdot \rchi_{\bbS_n}(\dcalP(n)) \in \ZZ[x_1,\ldots,x_{m+n}]^{\bbS_{m+n}}.\]
Monoids in the category of symmetric collections with respect to the monoidal product $\diamond$ are known under the name \emph{twisted associative algebras} (see e.g.~\cite{Species}), the same  categories with  the action of symmetric groups forgotten are known as \emph{permutads} after~\cite{Loday_Ronco} and as \emph{shuffle algebras} after~\cite{DK::Shuffle::Pattern}. All these categories have a lot of common properties with the category of associative algebras. In particular, one can easily define a notion of a quadratic and Koszul twisted associative algebra/permutad/shuffle algebra.
Our second major finding relates the PBW property with the Koszul property of operads and twisted associative algebras: 
\begin{theorem*}[Theorem~\ref{cor::Perm::UP::Koszul}]
The universal enveloping functor $\UP$ associated with a symmetric (quadratic) Koszul operad $\calP$ satisfies PBW if and only if the twisted associated algebra $\dcalP^{!}$ associated with the  corresponding Koszul dual operad $\calP^{!}$ is Koszul. 
\end{theorem*}
Moreover, we explain that the PBW property of $\UP$ for a Koszul operad $\calP$ implies that the corresponding universal enveloping algebra $\UP(V)$ is a nonhomogeneous Koszul algebra for any $\calP$-algebra $V$ (Corollary~\ref{cor::UP::Koszul}). This observation crucially simplifies the homology theory and homological complexes associated with the abelian category of $V$-modules.
In particular, the corresponding derived functors for modules over $\UP(V)$ coincide in some cases with the homological theories described in~\cite{Balavoine,Miles_cohomology,Goerss_Hopkins}.

We also find a relationship between the generating series of $\bbS$-characters of $\calP$ and the generating series of $\bbS$-characters of $\UP$ whenever $\UP$ satisfies PBW (Theorem~\ref{thm:Hilb::ser::PBW}). Finally, we work out necessary definitions of Gr\"obner bases for the colored operads under consideration and prove the main sufficient condition for PBW of $\UP$:
\begin{theorem*}[Theorem~\ref{thm::U:PBW}]
If a symmetric operad $\calP$ admits a Gr\"obner basis (in the sense of~\cite{DK::Grob}) whose leading monomials are trees with no branches growing to the right from any of the vertices then the universal enveloping functor $\UP$ satisfies PBW. 
\end{theorem*}
We illustrate our theory on the following list of examples that satisfy PBW:
\begin{itemize}
	\item the operad $\Lie$ of Lie algebras -- Examples~\ref{ex::Lie},~\ref{ex::Ass::PBW}; 
	\item the operad $\Com$ of commutative algebras -- Example~\ref{ex::Com},
	\item the operad $\Ass$ of associative algebras -- Examples~\ref{ex::Ass},~\ref{ex::Ass::PBW};
	\item the operad $\Lie_2$ of pairs of compatible Lie brackets -- Section~\S\ref{ex::Lie2};
	\item the Koszul dual operad to the operad $\calO_{A}$ -- Example~\ref{ex::Alg2Op};
	\item the operad $\PreLie$ of PreLie algebras -- Section~\S\ref{ex::PreLie};
	\item the operad $\Zinb$ of Zinbiel algebras -- Section~\S\ref{ex::Zinb}.
\end{itemize}
and the following list of examples that do not satisfy PBW:
\begin{itemize}
		\item the operad $\Pois$ of Poisson algebras -- Example~\ref{ex::Poisson}, Section~\S\ref{sec::Pois};
		\item the operad $\Perm=(\PreLie)^{!}$ of permutative algebras -- Section~\S\ref{ex::Perm};
		\item the operad $\Leib$ of Leibniz algebras -- Section~\S\ref{ex::Leib};
\end{itemize}
We hope that the methods we suggest below are enough to decide for almost all Koszul operads mentioned in~\cite{Zinbiel} whether the corresponding universal enveloping functor satisfies PBW.

\subsection{Structure of the paper}

Section~\S\ref{sec::UP} provides several equivalent definitions of the universal enveloping algebra and the universal enveloping functor.
In particular, we recall the pictorial definition suggested in~\cite{Ginz_Kapranov}, the definition via adjoint functors, and give an explicit construction.
\\
Section~\S\ref{sec::PBW} is devoted to different equivalent definitions of the PBW property what follow  a natural definition of the derived associative universal enveloping algebra presented in Section~\S\ref{sec::DUP}.
\\
Section~\S\ref{sec::Perm::main} is the core section containing the homological algebra around the universal enveloping functor.
In particular, we end up with the bar-cobar resolutions of this functor in Section~\ref{sec::perm::Bar::Cobar}.
\\
Section~\S\ref{sec::PBW::UP} contains corollaries of the preceding Section~\S\ref{sec::Perm::main} that lead to different necessary and sufficient conditions of the PBW property that are of most interest for applications. 
\\
Certain relatively simple examples are outlined in Section~\S\ref{sec::examples}.  

\section*{Acknowledgement}
This paper appears after a question of Vladimir Hinich on the existence of the criterion of PBW property of universal enveloping algebras and I would like to thank him for this question which I hope is completely solved in this note.
I would like to thank Vladimir Dotsenko for his interest in the first version of this note and for suggesting many improvements. 
I appreciate the contribution of my colleagues Ian Marshall and Chris Brav to improving the English text of this paper.
%pointing out many disadvantages in it. In particular, he corrected the important terminology of this paper.

My research was carried out within the HSE University Basic Research Program
and funded (jointly) by the Russian Academic Excellence Project '5-100'. 
The results of  Section~\S\ref{sec::Perm::main}  have been obtained under support of the RSF grant No.19-11-00275. % A.Okunkov RNF
The results of Section~\S\ref{sec::examples} have been obtained under support of RFBR grant No.{19-02-00815}.
I am a Young Russian Mathematics award winner and would like to thank its sponsors and jury.

%\newpage

\tableofcontents

\section*{Notation}
For simplicity, we deal with vector spaces (algebras, operads e.t.c) over a 
 field $\kk$ of zero characteristic. However, most of the results remain valid for fields of positive characteristic.

An $\mathbb{S}$-collection  (a symmetric collection) $\calP$ is a collection  $\{\calP(n)|n\geq 1\}$ with a given $\bbS_n$-action on $\calP(n)$. Using the Schur-Weyl duality isomorphism each $\bbS$-collection $\calP$ defines a polynomial functor that maps a vector space $V$ to the sum $\oplus_{n\geq 1} \calP(n)\otimes_{\bbS_{n}} V^{n}$.

Many notations are used uniformly in different categories such as operads, colored operads, twisted associative algebras, permutads and algebras. Most of them coincide with the one used in the textbook on operads~\cite{LV}. In particular,
 $\calF$ stands for the free object in the appropriate category, $\calB$ is the bar-construction, $\Omega$ denotes the cobar-construction,  $\calP\mapsto \calP^{!}$ is the contravariant functor, and  $\calP\mapsto \calP^{\antish}$ is the covariant Koszul duality  functor in the appropriate category. It should be clear from  context which category is considered. However, in order to highlight the underlying category we use an additional lower superscript.
In particular,
\begin{itemize}
	\item $\Omega_{\oprd}(\calP)$ stands for the cobar construction of a symmetric operad $\calP$;
	\item $\calB_{1\ttt2\oprd}(\calP) = \calB_{1\ttt2}(\calP)$ denotes the bar construction of a colored cooperad $\calP$ on two colors, such that the number of inputs and outputs of the second color in $\calP$ coincide;
	\item $\calP^{\antish}_{\perm}$ stands for the Koszul dual twisted associative coalgebra of a (quadratic) twisted associative algebra $\calP$;
	\item $A^{!}_{\alg}$ denotes the Koszul dual algebra associated with an associative algebra $A$;
	\item $\calF_{\mathsf{Sh}\ttt\alg}(V)$ denotes the free shuffle algebra generated by a collection $V$.
\end{itemize}
By \emph{right module over a symmetric operad $\calP$} we mean $\bbS$-collection  $\calM:=\{\calM(n)|n\in \NN\}$ together with the composition maps 
\[\circ:\calM(k)\otimes \calP(m_1)\otimes\ldots\otimes \calP(m_k) \rightarrow \calM(m_1+\ldots+m_k)\] that are associative in the standard operadic sense.
Similarly, \emph{a left $\calP$-module} is an $\bbS$-collection $\calM(n)$ equipped with composition rules  \[\circ:\calP(k)\otimes\calM(m_1)\otimes\ldots\calM(m_k)\rightarrow \calM(m_1+\ldots+m_k)\]
satisfying similar coherence conditions. In some cases it is better to forget the additional grading by operations, in which case the left module  becomes just an algebra over an operad.
Note that the category of right $\calP$-modules is an abelian category and the category of left $\calP$-modules is not abelian.

We denote by $\uk$ the left/right trivial module with $\uk(1)=\kk$ and $\uk(n)=0$ for all $n>1$.

The operad $s$ -- is the endomorphism operad of the homologically shifted vector space $\uk[-1]$, such that the structure of an algebra (a left modules) over the operad $s\calP$ on a vector space $V$ is the same as the structure of a $\calP$-algebra on a homologically shifted space $V[1]$.

\section{Associative universal enveloping algebra of an algebra over an operad}
\label{sec::UP}
	
\subsection{Classical pictorial definition by Ginzburg and Kapranov}
\label{sec::UP::GK}

While describing the notion of a module over an algebra over an operad,
V.Ginzburg and M.Kapranov introduced  the notion of  \emph{universal enveloping algebra} $\UP(V)$ for an algebra $V$ over an operad $\calP$. 
\begin{definition}
\label{def::UP::GK}	
(\cite{Ginz_Kapranov} Definition 1.6.4 on page~225)
The universal enveloping algebra $\UP(V)$ of a $\calP$-algebra $V$  is generated by symbols:
\[
\left\{
\left.
X(\gamma;a_1,\ldots,a_n):=
\begin{tikzpicture}[scale=0.6]
\node[ext] (v0) at (0,0) {\tiny{${\gamma}$}};
\coordinate (w0) at (-1.5,1);
\node (s1) at (-0.7,1.3) {\tiny${v_1}$};
\node (s2) at (0.2,1.3) {\tiny${v_2}$};
\node (s3) at (1.7,1.3) {\tiny${v_n}$};
\node (w) at (0.7,1) {\small${\ldots}$};
\draw[dotted] (v0) edge (w0);
\draw (v0) edge (s1) edge (s2) edge (s3);
\coordinate (u) at (1.5,-1);
\draw[dotted] (v0) edge (u); 
\end{tikzpicture}
\ \right| \ \gamma\in \calP(n+1), \ v_i \in V, i =1\ldots n 
\right\} 
\]
subject to  the condition of  being multilinear with respect to each argument, and to the following identifications being valid for all $\sigma\in\bbS_m\subset\bbS_{m+1}$, $v_i,w_j\in V$, $\gamma\in\calP(n)$,$\delta\in \calP(m)$:
\[
\begin{array}{ccc}
{
\begin{tikzpicture}[scale=0.6]
\node[ext] (v0) at (0,0) {\tiny{${\sigma(\gamma)}$}};
\coordinate (w0) at (-1.5,1);
\node (s1) at (-0.7,1.3) {\tiny${v_1}$};
\node (s2) at (0.2,1.3) {\tiny${v_2}$};
\node (s3) at (1.7,1.3) {\tiny${v_{m}}$};
\node (w) at (0.7,1) {\small${\ldots}$};
\draw[dotted] (v0) edge (w0);
\draw (v0) edge (s1) edge (s2) edge (s3);
\coordinate (u) at (1.5,-1);
\draw[dotted] (v0) edge (u); 
\end{tikzpicture}} & {=}&
{
\begin{tikzpicture}[scale=0.6]
\node[ext] (v0) at (0,0) {\tiny{${\gamma}$}};
\coordinate (w0) at (-1.5,1);
\node (s1) at (-0.7,1.3) {\tiny${v_{\sigma(1)}}$};
\node (s2) at (0.2,1.3) {\tiny${v_{\sigma(2)}}$};
\node (s3) at (1.7,1.3) {\tiny${v_{\sigma(m)}}$};
\node (w) at (0.7,1) {\small${\ldots}$};
\draw[dotted] (v0) edge (w0);
\draw (v0) edge (s1) edge (s2) edge (s3);
\coordinate (u) at (1.5,-1);
\draw[dotted] (v0) edge (u); 
\end{tikzpicture},} \\
{X(\sigma(\gamma);v_1,\ldots,v_m) } &  & {X(\gamma;v_{\sigma(1)},\ldots,v_{\sigma(m)}),} 
\end{array}
\]
and
\[
\begin{array}{ccc}
{
\begin{tikzpicture}[scale=0.6]
\node[ext] (v0) at (0,0) {\tiny{${\gamma\circ_{i}\delta}$}};
\coordinate (w0) at (-1.5,1);
\node (s1) at (-0.7,1.3) {\tiny${v_1}$};
\node (s2) at (0.2,1.3) {\tiny${v_2}$};
\node (s3) at (1.7,1.3) {\tiny${v_{m+n-1}}$};
\node (w) at (0.7,1) {\small${\ldots}$};
\draw[dotted] (v0) edge (w0);
\draw (v0) edge (s1) edge (s2) edge (s3);
\coordinate (u) at (1.5,-1);
\draw[dotted] (v0) edge (u); 
\end{tikzpicture} } & {=} &
{\begin{tikzpicture}[scale=0.6]
\node[ext] (v0) at (0,0) {\tiny{${\gamma}$}};
\coordinate (w0) at (-1.5,1);
\node (s1) at (-0.7,1.3) {\tiny${v_1}$};
\node (s2) at (0,1.3) {\tiny{\dots}};
\node (s3) at (2,1.3) {\tiny${v_{m+n-1}}$};
\node[ext] (s4) at (0.3,1.4) {\tiny{$\delta$}};
\node (w) at (1.3,1) {\tiny${\ldots}$};
\node (t1) at (-0.2,2.2) {\tiny{$v_{i+1}$}};
\node (t2) at (0.5,2.2) {\tiny{\dots}};
\node (t3) at (1.5,2.2) {\tiny{$v_{i+m}$}};
\draw (s4) edge (t1) edge (t3);
\draw[dotted] (v0) edge (w0);
\draw (v0) edge (s1) edge (s3) edge (s4);
\coordinate (u) at (1.5,-1);
\draw[dotted] (v0) edge (u); 
\end{tikzpicture} .} \\
{X(\gamma\circ_i\delta; v_1,\ldots, v_{m+n-1}) }& & X(\gamma;v_1,\ldots,\delta(v_{i},\ldots,v_{i+m-1}),\ldots,v_{m+n-1}).
\end{array}	 
\]
 Multiplication is  defined by  concatenation of the dotted paths
\[
\begin{array}{ccccc}
{
\begin{tikzpicture}[scale=0.6]
\node[ext] (v0) at (0,0) {\tiny{${\gamma}$}};
\coordinate (w0) at (-1.5,1);
\node (s1) at (-0.7,1.3) {\tiny${v_1}$};
\node (s2) at (0.2,1.3) {\tiny${v_2}$};
\node (s3) at (1.7,1.3) {\tiny${v_n}$};
\node (w) at (0.7,1) {\tiny${\ldots}$};
\draw[dotted] (v0) edge (w0);
\draw (v0) edge (s1) edge (s2) edge (s3);
\coordinate (u) at (1.5,-1);
\draw[dotted] (v0) edge (u); 
\end{tikzpicture}
}
& {\cdot}
&{
\begin{tikzpicture}[scale=0.6]
\node[ext] (v0) at (0,0) {\tiny{${\delta}$}};
\coordinate (w0) at (-1.5,1);
\node (s1) at (-0.7,1.3) {\tiny${w_1}$};
\node (s2) at (0.2,1.3) {\tiny${w_2}$};
\node (s3) at (1.7,1.3) {\tiny${w_m}$};
\node (w) at (0.7,1) {\tiny${\ldots}$};
\draw[dotted] (v0) edge (w0);
\draw (v0) edge (s1) edge (s2) edge (s3);
\coordinate (u) at (1.5,-1);
\draw[dotted] (v0) edge (u); 
\end{tikzpicture}
}
&
{:=
\begin{tikzpicture}[scale=0.6]
\node[ext] (v1) at (0,-1) {\tiny{$\delta$}};
\node[ext] (v2) at (-1,1) {\tiny{$\gamma$}};
\coordinate (r) at (1.5,-2);
\coordinate (w1) at (-1.5,2.5);
\node (w2) at (-0.5,2.5) {\tiny{$v_1$}};
\node (ww2) at (0.0,2.5) {\tiny{$\ldots$}};
\node (ww3) at (0.5,2.5) {\tiny{$v_n$}};
\node (w3) at (0.5,1) {\tiny{$w_1$}};
\node (wu2) at (1,1) {\tiny{$\ldots$}};
\node (wu3) at (1.5,1) {\tiny{$w_m$}};
\draw[dotted] (v1) edge (r);
\draw[dotted] (v2) edge (v1);
\draw (v1)  edge (w3) edge (wu2) edge (wu3);
\draw[dotted] (w1) edge (v2);
\draw (v2) edge (w2) edge (ww2) edge (ww3);
\end{tikzpicture} 
=
} &
{\begin{tikzpicture}[scale=0.6]
\node[ext] (v0) at (0,0) {\tiny{${\delta\circ_1\gamma}$}};
\coordinate (w0) at (-1.5,1);
\node (s1) at (-0.7,1.3) {\tiny${v_1}$};
\node (s2) at (0.2,1.3) {\tiny${v_2}$};
\node (s3) at (1.7,1.3) {\tiny${w_m}$};
\node (w) at (0.7,1) {\tiny${\ldots}$};
\draw[dotted] (v0) edge (w0);
\draw (v0) edge (s1) edge (s2) edge (s3);
\coordinate (u) at (1.5,-1);
\draw[dotted] (v0) edge (u); 
\end{tikzpicture} }\\
X(\gamma;v_1,\ldots,v_n) & \cdot & X(\delta;w_1,\ldots,w_m) & & X(\delta\circ_1\gamma; v_1,\ldots,w_m) 
\end{array}
\]
\end{definition}

\subsection{Definition via colored $1\ttt2$-operads}
\label{sec::UP::12}

Let us rephrase Definition~\ref{def::UP::GK} of the universal enveloping algebra using the language of colored operads on two colors $\{1,2\}$.  Straight lines  will represent  the first color $1$, and dotted arrows correspond to the second color $2$. 
The space of operations of the colored operad $\calQ$ with $m$ inputs of the first color $\{1\}$ and $n$ inputs of the second color $\{2\}$ and with an output  of  color $c\in \{1,2\}$, is denoted by $\calQ(m,n)^{c}$.
\begin{definition}	
\label{dfn::P+}	
 We associate a colored operad $\calP_{+}$ on two colors $\{1,2\}$ to each symmetric algebraic operad $\calP=\cup_{n\geq 1} \calP(n)$ in the following way:
 \begin{itemize}
 \setlength{\itemsep}{-0.2em}	
 \item The suboperad $\calP_{+}(\ttt,\ttt)^{1}$ spanned by elements with outputs of the first color is isomorphic to $\calP$. That is $\calP_{+}(m,n)^{1}=0$ whenever $n>1$ and $\calP_{+}(\ttt,0)^{1}\simeq \calP(\ttt)$.
 	\item 
 The space of operations $\calP_{+}(m,n)^{2}$ with  output of the second color is empty whenever $n\neq 1$ and to each operation $\gamma\in\calP(n+1)$ we assign a unique operation $\partial\gamma\in\calP_+(n,1)^2$ by coloring the first input and the output in the second color. This recoloring operation is called the \underline{derivative}:
 \begin{equation}
 \label{eq::pic::derivative}
 \begin{tikzpicture}[scale=0.6]
 \node[ext] (v0) at (0,0) {\tiny{${\gamma}$}};
 \node (w0) at (-1.5,1) {\tiny$1$};
 \node (s1) at (-0.7,1.3) {\tiny${2}$};
 \node (s2) at (0.2,1.3) {\tiny${3}$};
 \node (s3) at (1.7,1.3) {\tiny${n+1}$};
 \node (w) at (0.7,1) {\tiny${\ldots}$};
 \draw (v0) edge (w0);
 \draw (v0) edge (s1) edge (s2) edge (s3);
 \coordinate (u) at (0,-1);
 \draw (v0) edge (u); 
 \end{tikzpicture}
 \stackrel{\partial}{\mapsto}
 \begin{tikzpicture}[scale=0.6]
\node[ext] (v0) at (0,0) {\tiny{${\gamma}$}};
\node (w0) at (-1.5,1) {\tiny$\bar{1}$};
\node (s1) at (-0.7,1.3) {\tiny${1}$};
\node (s2) at (0.2,1.3) {\tiny${2}$};
\node (s3) at (1.7,1.3) {\tiny${n}$};
\node (w) at (0.7,1) {\tiny${\ldots}$};
\draw[dotted] (v0) edge (w0);
\draw (v0) edge (s1) edge (s2) edge (s3);
\coordinate (u) at (0,-1);
\draw[dotted] (v0) edge (u); 
\end{tikzpicture} 
 \end{equation}
\item 
The derivative map~\eqref{eq::pic::derivative} commutes with compositions:
\[ \forall\gamma,\delta\in \calP \text{ we have }
\begin{array}{c}
(\partial\gamma )\circ^1_i \delta =  \partial (\gamma\circ_{i+1}\delta); \\
(\partial\gamma) \dcirc^2_1 (\partial \delta) = \partial (\gamma\circ_1 \delta )
\end{array} 
\]
and this uniquelly defines all composition rules in $\calP_{+}$. Here $\circ^1$ corresponds to the composition through the first color and $\dcirc^2$ is the composition through the second dotted color.
\end{itemize}
\end{definition}
\begin{remark}
	The operation of marking one input in a combinatorial species (=symmetric collection) $\calP$ is  the \emph{derivative} of a species and is denoted  $\dcalP$ (see e.g.~\cite{Species}).
\end{remark}
\begin{definition}
\label{def::12-oper}
  A  colored operad with two colors whose spaces of operations differ from zero only if the number of inputs and outputs of the second color coincide will be called a \underline{symmetric $1\ttt2$-operad}.
\end{definition}
In particular, the colored operad $\calP_+$ and the colored suboperad $\calP_{+}^{2}$ with  output of the color $\{2\}$ are  examples of  $1\ttt2$-operads.
\begin{remark}
	An algebra over $\calP_{+}$ is a pair $(V,M)$, where $V$ is an algebra over $\calP$ and $M$ is a module over $V$. This is a paraphrase  of the definition presented by Ginzburg and Kapranov in section~1.6 of~\cite{Ginz_Kapranov}.
\end{remark}
Consider the natural embedding  $\imath:\calP \to \calP_{+}$ of operads on $2$ colors where the symmetric operad $\calP$ is considered to have empty space of operations of the second color and the restriction functor $F^{\calP}$ from the corresponding category of $\calP_{+}$-algebras to the cartesian product of the category of $\calP$-algebras and the category of vector spaces. The left adjoint functor $F^{\calP}_{!}$ exists thanks to the following coequalizer construction:
\[
F^{\calP}_{!}:
\begin{array}{ccl}
\calP\ttt\text{algebras}\times \mathcal{V}\text{ector spaces}
& \rightarrow & \calP_{+}\ttt\text{algebras} \\
(V,M) & \mapsto & (V, \calP_+(\ttt,1)^2\circ_{\calP} (V,M)) 
\end{array}
\]
where $\calP_+(\ttt,1)^2\circ_{\calP} (V,M)$ is the coequalizer of the diagram
\begin{equation}
\label{eq::UP::equalizer}
\begin{tikzcd}
& \calP_{+}^{2} \circ (\calP \circ V,M) \arrow[dr] & \\
\calP_{+}^{2}\circ \calP \circ (V,M) \arrow[ur] \arrow[r] & (\calP_{+}^{2} \circ \calP) \circ (V,M) \arrow[r] & \calP_{+}^{2}(V,M)  
\end{tikzcd}
\end{equation}
\begin{proposition}
A composition through the second color defines an associative multiplication on the coequalizer $\calP_+(\ttt,1)^2\circ_{\calP} (V,\kk)$ which is canonically isomorphic to 
	\emph{the universal enveloping algebra} $\UP(V)$ of a $\calP$-algebra $V$ introduced by Ginzburg and Kapranov (Definition~\ref{def::UP::GK} above).
\end{proposition}
We refer to~\cite{Miles_cohomology} where some definitions of the universal enveloping are given in a more systematic way.

\subsection{Explicit construction of the universal enveloping algebra}
\label{sec::UP::Handmade}

Let us first mention the following simple but curious fact:
\begin{proposition}
\label{thm::exact::P_plus}	
	The functor ${+}:\calP\rightarrow \calP_{+}$ is exact.
\end{proposition}
\begin{proof}
This functor is exact on the level of underlying vector spaces of operations of given arity. Indeed, both  the functors $\calP\mapsto \calP$ and $\calP\mapsto \dcalP$ change the algebraic structure, but do not change the underlying vector space.
\end{proof}
Note that, moreover, the functor ${+}:\calP\mapsto\calP_{+}$ maps free operads to free operads. 
\begin{corollary}
\label{cor::+::free}	
	Suppose that the operad $\calP$ is generated by an $\bbS$-collection $\calV$ subject to an $\bbS$-collection of relations $\calR$. Then the $1\ttt2$-operad $\calP_+$ is generated by $\calV\cup \partial \calV$ subject to relations $\calR \cup \partial \calR$.
\end{corollary}
This leads to the \emph{handmade} description of the universal enveloping (associative) algebra $\UP(V)$. For simplicity, we explain  the case of operads generated by binary operations subject to quadratic relations, but one can easily generalize Corollary~\ref{prp::+Ex} to  generators and relations with arbitrary arities.
\begin{corollary}
\label{prp::+Ex}
 Suppose that $\calP$ is a symmetric operad generated by binary operations $\corolla{\gamma_1},\ldots,\corolla{\gamma_k}$ (that form a basis of the vector space $\calP(2)$)  and subject to the following basic set of quadratic relations indexed by the upper index $s\in S$:
\[
\sum_{i,j=1}^{k}
\left(
a_{ij}^{s}
\lltree{\gamma_i}{\gamma_j}
+
b_{ij}^{s}
\lrtree{\gamma_i}{\gamma_j}
+
c_{ij}^{s}
\rrtree{\gamma_i}{\gamma_j}
\right)
= 0.
\]
Then the universal enveloping algebra $\UP(V)$ of a $\calP$-algebra $V$ is a unital associative algebra generated by $k$ copies of $V$, where the $i$-th embedding of $V$ into the set of generators is denoted by $v\mapsto \partial\gamma_i(v)$  subject to the following set of (quadratic-linear) relations indexed by pairs of elements $v,w\in V$ and the index $s\in S$:
\[
\sum_{i,j=1}^{k}\left(
a_{ij}^{s} \partial\gamma_i(w)\cdot \partial\gamma_j(v) + b_{ij}^{s} \partial\gamma_i(v)\cdot \partial\gamma_j(w) + c_{ij}^{s} \partial\gamma_i(\gamma_j(v,w)) \right)=0
\]  
\end{corollary}

\begin{example}
\label{ex::Lie}	
 The operad $\Lie$ is generated by a unique anti-symmetric generator called the Lie bracket subject to the following relation:
\[
\begin{tikzpicture}[scale=0.3]
\node[ext] (v1) at (0,-1) {};
\node[ext] (v2) at (-1,1) {};
\coordinate (r) at (0,-2);
\node (w1) at (-1.5,2.5) {${\tiny{1}}$};
\node (w2) at (-0.5,2.5) {${\tiny{2}}$};
\node (w3) at (0.5,1) {${\tiny{3}}$};
\draw (v1) edge (r);
\draw (v1) edge (v2) edge (w3);
\draw (v2) edge (w1) edge (w2);
\end{tikzpicture} +
\begin{tikzpicture}[scale=0.3]
\node[ext] (v1) at (0,-1) {};
\node[ext] (v2) at (-1,1) {};
\coordinate (r) at (0,-2);
\node (w1) at (-1.5,2.5) {${\tiny{2}}$};
\node (w2) at (-0.5,2.5) {${\tiny{3}}$};
\node (w3) at (0.5,1) {${\tiny{1}}$};
\draw (v1) edge (r);
\draw (v1) edge (v2) edge (w3);
\draw (v2) edge (w1) edge (w2);
\end{tikzpicture}
 +
\begin{tikzpicture}[scale=0.3]
\node[ext] (v1) at (0,-1) {};
\node[ext] (v2) at (-1,1) {};
\coordinate (r) at (0,-2);
\node (w1) at (-1.5,2.5) {${\tiny{3}}$};
\node (w2) at (-0.5,2.5) {${\tiny{1}}$};
\node (w3) at (0.5,1) {${\tiny{2}}$};
\draw (v1) edge (r);
\draw (v1) edge (v2) edge (w3);
\draw (v2) edge (w1) edge (w2);
\end{tikzpicture}
=
\begin{tikzpicture}[scale=0.3]
\node[ext] (v1) at (0,-1) {};
\node[ext] (v2) at (-1,1) {};
\coordinate (r) at (0,-2);
\node (w1) at (-1.5,2.5) {${\tiny{1}}$};
\node (w2) at (-0.5,2.5) {${\tiny{2}}$};
\node (w3) at (0.5,1) {${\tiny{3}}$};
\draw (v1) edge (r);
\draw (v1) edge (v2) edge (w3);
\draw (v2) edge (w1) edge (w2);
\end{tikzpicture}
 -
\begin{tikzpicture}[scale=0.3]
\node[ext] (v1) at (0,-1) {};
\node[ext] (v2) at (-1,1) {};
\coordinate (r) at (0,-2);
\node (w1) at (-1.5,2.5) {${\tiny{1}}$};
\node (w2) at (-0.5,2.5) {${\tiny{3}}$};
\node (w3) at (0.5,1) {${\tiny{2}}$};
\draw (v1) edge (r);
\draw (v1) edge (v2) edge (w3);
\draw (v2) edge (w1) edge (w2);
\end{tikzpicture}
-
\begin{tikzpicture}[scale=0.3]
\node[ext] (v1) at (0,-1) {};
\node[ext] (v2) at (1,1) {};
\coordinate (r) at (0,-2);
\node (w1) at (1.5,2.5) {${\tiny{3}}$};
\node (w2) at (0.5,2.5) {${\tiny{2}}$};
\node (w3) at (-0.5,1) {${\tiny{1}}$};
\draw (v1) edge (r);
\draw (v1) edge (v2) edge (w3);
\draw (v2) edge (w1) edge (w2);
\end{tikzpicture}
= 0
\]
The classical universal enveloping algebra of a Lie algebra $\frg$ coincides with the operadic universal enveloping algebra,
\[
\UU_{\Lie}(\frg) = \kk\left\langle
\frg | g\otimes h - h \otimes g - [g,h] 
\right\rangle.
\]
\end{example}
\begin{example}
\label{ex::Com}	
	The operad $\Com$ of commutative algebras is generated by a symmetric operation called multiplication, subject to the following relations:
	\[
	\lltree{}{} = \lrtree{}{} = \rrtree{}{}.
	\]
The universal enveloping algebra  $\UU_{\Com}(V)$ of a commutative nonunital algebra $V$ is isomorphic to $\kk \oplus V$.	
\end{example}

\begin{example}
\label{ex::Ass}
The symmetric operad $\Ass$ of associative algebras is generated by two binary operations; that is, a multiplication ${\tiny{\oplus}}(a,b):= ab$ and its opposite ${\tiny{\ominus}}(a,b):=ba$. The associativity relation $(ab)c=a(bc)$ is  viewed as $6$ linearly independent quadratic relations (each expression below is a pair of  relations that differ by  changing  each  operation to the opposite one):
\[
\lltree{\pm}{\pm} = \rrtree{\pm}{\pm},
\
\lltree{\pm}{\mp} = \lrtree{\mp}{\pm},
\
\lrtree{\pm}{\pm}=\rrtree{\pm}{\mp}.
\]
Thus the corresponding universal enveloping of an associative algebra $V$ has the following presentation:
\begin{equation}
\label{eq::UAss}
\UU_{\Ass}(V):= \kk\left\langle V\oplus V 
\left|
\begin{array}{c}
(v)_{+} (w)_{+} = (vw)_{+}, \\
(v)_{+} (w)_{-} = (w)_{-} (v)_{+} \\
(v)_{-} (w)_{-} = (wv)_{-} \\
\end{array}
\right.\right\rangle
\simeq (\kk\oplus V)\otimes (\kk \oplus  V^{\op}),
\end{equation}
where $(v)_{+}$ (resp. $(v)_{-}$ denotes the more complicated notation $\partial\oplus(v)$ (resp. $\partial\ominus(v)$).
\end{example}
Many more nontrivial examples  will be given in Section~\S\ref{sec::examples} below. 

\begin{remark}
	As one can see from  Example~\ref{ex::Ass}, we are dealing with nonunital algebras over operads. On the level of operads this means that we consider operads without $0$-ary operations. The functor ${+}:\calP\mapsto\calP_{+}$ can be easily generalized to operads with nontrivial $0$-ary operations. However, it is worth mentioning that the universal enveloping of an algebra over an operad $\calP$ with and  without a $0$-ary operation are different from one another as one can see from the examples of the operad $\Ass$ (considered in~\cite{Ginz_Kapranov}) and the operad~$\Pois$ considered in Example~\ref{ex::Poisson} without $0$-ary operations and in~\cite{Umirbaev_Poisson} with a unit element.
\end{remark}

\section{PBW-property and derived universal enveloping algebras}
\subsection{PBW property}
\label{sec::PBW}
Let $\calP$ be a symmetric operad. Let $J(\calP)\subset \calP$ be the augmentation ideal, i.e. $\calP/(J(\calP))$ is isomorphic to the trivial one-dimensional operad spanned by the  identity element $\mathbb{1}\in\calP(1)$. For an operad given by generators and relations,  $J(\calP)$ is the ideal generated by the generators of $\calP$ and $\calP/J(\calP)\simeq \uk$.
For each vector space $V$ we can define a trivial $\calP$ structure on $V$ by setting the $J(\calP)$-action to be zero. In other words, it is enough to say that for all generators $\gamma\in J(\calP(n))\subset \calP$ and for all collections $v_1,\ldots,v_n\in V$ we pose $\gamma(v_1,\ldots,v_n):=0$. We will denote the corresponding $\calP$-algebra by $V_0$.

As mentioned in Corollary~\ref{prp::+Ex}  the generators of the associative universal enveloping algebra $\UP(V)$ are indexed by generators of $\calP$ and elements of $V$ (if $\calP$ is binary generated) and by powers of $V$ in the general case.
This leads to the so-called \emph{PBW-filtration}, which is the functorial in $V$ filtration of $\UP(V)$ by the degree of generators.
\begin{definition}
\label{def::PBW1}	
We say that the associative universal enveloping functor $\UP$ satisfies the \emph{PBW}-property iff for any $\calP$-algebra $V$ there is a functorial in $V$  isomorphism of associative algebras:
\begin{equation}
\label{eq::grPBW}
\gr^{PBW} \UP(V) \simeq \UP(V_0)
\end{equation}
called the \underline{PBW-isomorphism}.
\end{definition}	
Note that since the space of operations in an operad is graded by arity one always  has a \emph{dilation} automorphism $\hbar$ of an operad  that multiplies the space of $n$-ary operations by $\hbar^{n-1}$. 
The PBW property for operads generated by elements of arity at least $2$ can be restated using the language of deformation theory.
For each $\calP$-algebra $V$, let us denote by $V_{\hbar}$ the corresponding  $\calP$-algebra structure on $V$ dilated by $\hbar$. That is, we define  $\gamma^{\hbar}(v_1,\ldots,v_n):=\hbar^{n-1}\gamma(v_1,\ldots,v_n)$.
\begin{corollary}
The universal enveloping functor $\UP$ satisfies  the PBW-property 
if and only if the family of associative universal enveloping algebras $\UP(V_{\hbar})$ defines a flat deformation over $\kk[\hbar]$ of the universal enveloping algebra $\UP(V_{0})$.
\end{corollary}
\begin{proof}
The algebra $\UP(V_{\hbar})|_{\hbar=q}$ is isomorphic to $\UP(V)$ for $q\neq 0$ and is isomorphic to $\UP(V_0)$ for $\hbar=0$.
\end{proof}
The following important theorem is a simple corollary of the general PBW-property stated for the morphism of monads by V.\,Dotsenko and P.\,Tamaroff in~\cite{Dots::PBW}:
\begin{theorem}
\label{thm::PBW::dP::free}	
	The associative universal enveloping functor $\UP$ satisfies the PBW-property if and only if $\dcalP$ is a free right $\calP$-module.
\end{theorem}
\begin{proof}
Consider the map of colored operads $\calP\to \calP_{+}$ as a morphism of monads and  apply Theorem~2 of~\cite{Dots::PBW}, which states that $\UP$ admits PBW iff $\calP_+$ is a free right $\calP$-module. On the other hand, the operation with the output of the first color is isomorphic to the free $\calP$-module with one generator, so it is enough to check  freeness of the operations with the second color.
\end{proof}
\begin{corollary}
\label{cor::PBW::free::P}	
If  $\UP$ has the PBW-property then there is an isomorphism of $\bbS$-collections:
\[
\dcalP \simeq \UPZ \circ \calP
\]
where $\UPZ$ is the $\bbS$-collection assigned to the Schur functor $V\mapsto \UP(V_0)$.
Moreover, $\UPZ$ is isomorphic to $\dcalP\circ_{\calP}\uk$ (compare with the coequalizer~\eqref{eq::UP::equalizer}).  
\end{corollary}

The abbreviation PBW comes from the following famous Poincar\'e-Birkhoff-Witt theorem for Lie algebras:
\begin{example}
\label{ex::Lie::PBW}	
	The graded algebra associated to the PBW filtration of the universal enveloping algebra $\UU_{\Lie}(\frg)$ of a Lie algebra $\frg$ is isomorphic to the symmetric algebra $S(\frg)$.
	In other words, the functor $\UU_{\Lie}$ admits the PBW property.
\end{example}

\begin{example}
\label{ex::Ass::PBW}	
	The functors $\UU_{\Ass}:V\mapsto (\kk\oplus V)\otimes (\kk\oplus V^{\op})$ and $\UU_{\Com}:V\mapsto \kk\oplus V$ satisfy the PBW property.	
\end{example}
Let us also provide an immediate counterexample and, of course, there will be more in Section~\S\ref{sec::examples}.
\begin{example}
	Suppose that the (binary generated) operad $\calP$ is finite dimensional, in particular, $\calP(n)=0$ for $n>>0$. Then $\UP$ does not satisfy the PBW property because  comparison of degrees shows that $\dcalP$ may  be non-isomorphic to $\calQ\circ\calP$.
\end{example}

\subsection{Derived associative universal enveloping  algebra}
\label{sec::DUP}

Thanks to Theorem~\ref{thm::PBW::dP::free} we know that  the PBW property is not satisfied for $\UP$ if
$\dcalP$ is not a free right $\calP$-module and, consequently, there is no reason to have an equivalence of universal  enveloping algebras of equivalent differential graded $\calP$-algebras and equivalent dg-models of an algebraic operad $\calP$. The aim of this section is to give a correct homotopical definition of the notion of  derived associative universal enveloping algebra.

Indeed, the category of right modules over a given operad $\calP$ is abelian and hence the functor 
$$\ttt\circ_{\calP} V: \texttt{ right }\calP\ttt\texttt{modules} \longrightarrow \text{ Vector spaces }$$
is a right-exact functor between abelian categories that admits a left derived functor $\ttt\stackrel{L}{\circ}_{\calP} V$ since this category has enough free modules.
We end up with the following
\begin{definition}
	\label{prp::perm::UP}	
	The \emph{derived associative universal enveloping} $\UP^d(V)$ of a $\calP$-algebra $V$ is the associative algebra $\dcalP\stackrel{L}{\circ}_{\calP}  V$.
\end{definition}
 Hence, in order to get a description of the derived associative universal enveloping $\UP^d(V)$ one has to write down a free (projective) resolution  $U^{\udot}\circ\calP\twoheadrightarrow \dcalP$ in the category of right $\calP$-modules. Then for any $\calP$-algebra $V$ there exists a functorial in $V$ structure of  associative dg-algebra structure on $U^{\udot}(V):=\oplus_{n\geq 0} U^{\udot}(n)\otimes_{\bbS_{n}} V^{\otimes n}$ and $U^{\udot}(V)$ is a model for $\UP^d(V)$.
In particular, if $\UP$ satisfies the PBW  property we have a quasi-isomorphism $\UP^d(V)\twoheadrightarrow \UP(V)$.
%In general, the following proposition explains the nature of the derived universal enveloping.
\begin{proposition}
	The derived universal enveloping algebras $\UP^{d}(V)$ are equivalent for equivalent models of $\calP$ and $V$.
 That is,	if $\varphi:\calQ\rightarrow \calP$ is a quasi-isomorphism of symmetric operads, and $\psi:V\to V'$ is a quasi-isomorphism of (dg) $\calP$-algebras $V$, then one has a quasi-isomorphism of dg associative algebras:
	\[ \UU_{\calQ}^{d}(V) \stackrel{\varphi}{\rightarrow} \UU_{\calP}^{d}(V) \stackrel{\psi}{\rightarrow} \UP^{d}(V').\]
\end{proposition}
\begin{proof}
	Follows from the definition of the derived universal enveloping algebra.
\end{proof}
This leads to a universal model of the derived associative universal enveloping algebra.
\begin{corollary}
 If $\calP_{\infty}\to\calP$ is a quasifree resolution of an operad $\calP$ generated by elements of arity greater then $1$, then the corresponding universal enveloping functor $\UU_{\calP_{\infty}}$ satisfies  the PBW property and we come up with the (quasi)isomorphisms:
	\[ \UU_{\calP_{\infty}}(V)= \UU_{\calP_\infty}^d(V) \stackrel{quis}{\twoheadrightarrow} \UP^{d}(V)\] 
\end{corollary}
\begin{proof}
	The operad $\calP_{\infty}$ is free as an operad, consequently, the 2-colored operad $(\calP_{\infty})_+$ is free and, in particular, $\dcalP_\infty$ is a free right $\calP_\infty$-module. Therefore, thanks to Theorem~\ref{thm::PBW::dP::free}, the PBW property is satisfied.
\end{proof}

\section{Twisted associative algebras and $1\ttt2$-operads $\calP_{+}$}
\label{sec::Perm::main}

\subsection{Recollection on Twisted associative algebras/Permutads}
\label{sec::perm::def}

Recall that a collection of (possibly zero) representations of symmetric groups $\cup_{n\geq 1} \calQ(n)$ is called a symmetric collection (or a  $\bbS$-collection). The most standard tensor product on $\bbS$-collections that reflects the multiplication of symmetric functions is given by induction of representations for symmetric groups:
\[
(\calQ\diamond\calQ')(n):= \bigoplus_{k+m=n} Ind_{{\bbS}_{k}\times {\bbS_m}}^{\bbS_{n}} \calQ(k)\otimes \calQ'(m)
\]
\begin{definition}(\cite{Dotsenko_Bremner::Shuffle})
	A \emph{twisted associative algebra} is a monoid in the category of symmetric collections with respect to the tensor product $\diamond$.
	In other words, a \emph{twisted associative algebra} is a symmetric collection $\calQ$ together with an associative product $\mu:\calQ\diamond \calQ \to \calQ$.  
\end{definition}
Choosing  particular representatives by shuffle permutations for the equivalence classes $\bbS_{m+n}/(\bbS_m\times\bbS_n)$,  V.Dotsenko and the author introduced the algebraic object named a \emph{Shuffle algebra} in~\cite{DK::Shuffle::Pattern} which was independently discovered by  M.\,Ronco in~\cite{Ronco} under the name of \emph{Permutad}. The equivalence of permutads and shuffle algebras is explained in~\cite{Loday_Ronco}.

 We do not want to recall these notions in full details and refer the reader to the papers \cite[chapter 4]{Dotsenko_Bremner::Shuffle}, \cite{DK::Shuffle::Pattern}, \cite{Loday_Ronco} since we hope that the definition is clear from the following:

\begin{proposition}
\label{st::perm2oper}	
To each twisted associative algebra $\calQ$ one can assign a $1\ttt2$-operad (two-colored operad defined in~\ref{def::12-oper}) $\tilde{\calQ}$ such that
\[
\tilde{\calQ}(m,n)^{c}:= \begin{cases}
\calQ(m), \text{ if } n=1, c=2; \\
0, \text{ otherwise }
\end{cases}
\]
the unique nontrivial composition in the $1\ttt2$-operad $\tilde{\calQ}(k,1)^{2}\dcirc_1^2\tilde{\calQ}(m,1)^{2}\rightarrow \tilde{\calQ}(k+m,1)^{2}$  is defined by the multiplication $\diamond$ in the twisted associative algebra
and, moreover, this assignment defines a fully faithful embedding.
\end{proposition}
In particular, twisted associative algebras are particular cases of  colored operads.
Consequently, the notion of a free object as well as the homological algebra of twisted associative algebras is the same as for colored operads.
However, the restriction on the number of colored inputs and colored outputs in the colored operad $\tilde{\calQ}$ crucially simplifies the underlying algebraic structure as we will see later on from the examples and the constructions below.

We recall the (co)bar constructions for (colored) operads in the succeeding Section~\S\ref{sec::bar-cobar} and want to underline that 
we give a special name to the bar (cobar) construction $\calB_{\perm}(\calQ)$ (resp. $\Omega_{\perm}(\calQ)$ of a twisted associative algebra $\calQ$ by adding a lower superscript.

\begin{example}
If $\calP$ is a symmetric operad then the collection $\dcalP:=\cup_{n\geq 0}\calP_{+}(n,1)^{2}$ with the composition with respect to the second color forms a twisted associative algebra.  The $\bbS_n$ action on $\calP_{+}(n,1)^{2}$ comes from permutations of the inputs of the first color and is isomorphic to the restriction representation $Res^{\bbS_{n+1}}_{\bbS_n}\calP(n+1)$.
We will denote the corresponding twisted associative algebra by $\dcalP$ following the standard notation for the derivative for species (see e.g.~\cite{Species}).
\end{example}

\begin{example}
If $\calP$ is a symmetric operad such that the universal enveloping functor $\UP$ satisfies PBW, then the $\bbS$-collection $\UPZ:=\dcalP\circ_{\calP} \uk$ is a twisted associative algebra.

If $\UP$ does not satisfy PBW then the derived universal enveloping functor $\UU_{\calP}^{d,0}:=\dcalP\stackrel{L}{\circ}_{\calP} \uk$ is a twisted associative algebra.
\end{example}

\subsection{Recollection on bar-cobar constructions}
\label{sec::bar-cobar}
This subsection is technical and is added to the main body of the text in order to make the proofs more transparent. We refer to~\cite{Ginz_Kapranov} and to~\cite{LV} for the detailed description of Bar-Cobar construction and Koszul duality for operads and algebras, to~\cite{Dotsenko_Bremner::Shuffle} for the details on twisted associative algebras, to~\cite{DK::Grob} for the details on shuffle operads and to~\cite{Yau_Colored} for details on colored operads.

Recall that the cobar construction $\Omega_{\oprd}(\ttt)$ is a functor from the category of coagmented dg cooperads to the category of augmented dg operads that assigns to a cooperad $\calP^{\dual}$ the differential graded quasi-free operad $\Omega_{\oprd}(\calP^{\dual})$ generated by the shifted cooperad $s^{-1}\calP^{\dual}$. We follow the notation of~\cite{LV} where $s$ -- is the homologically shifted (co)operad $\Com$  whose space of $n$-ary cooperations is the sign representation shifted in homological degree $1-n$. That is $s$ is an endomorphism operad $\mathcal{E}nd(\uk[1])$ of an odd one-dimensional space $\uk[1]$.
The differential is the sum of the differential in $\calP^{\dual}$ and the cooperad structure in $\calP^{\dual}$:
\begin{equation}
\label{eq::Cobar::dif}
d:s^{-1}\calP^{\dual} \stackrel{d_{\calP}+s^{-1}\mu_{\calP}}{\longrightarrow} s^{-1}\calP^{\dual}[-1]\oplus s^{-2}\calP^{\dual}\otimes\calP^{\dual}
\end{equation}
Respectively, the bar construction $\calB_{\oprd}$ is the functor from the category of augmented dg-operads to the category of coaugmented conilpotent dg cooperads.

One of the crucial results of homological algebra is the adjunction of the bar and cobar constructions:
\begin{equation}
\label{eq::bar::cobar}
\Omega_{\oprd}\colon \ \left\{\text{conilpotent dgcooperads} \right\} 
\rightleftharpoons
\left\{\text{augmented dg-operads}\right\} \colon \calB_{\oprd}.
\end{equation}
In particular, there are canonical quasi-isomorphisms $\Omega_{\oprd}(\calB_{\oprd}(\calP)) \rightarrow \calP$ and $\calP^{\dual} \rightarrow \calB_{\oprd}(\Omega_{\oprd}(\calP^{\dual}))$ for all finitely (co)generated  dg operads $\calP$ and dg cooperads $\calP^{\dual}$.

The elements of $\Omega_{\oprd}(\calP^{\dual})$ (respectively of $\calB_{\oprd}(\calP)$) are linear combinations of operadic trees whose vertices are labelled by elements of $\calP^{\dual}$ and the differential is given by vertex splitting morphism described above~\eqref{eq::Cobar::dif}. 
We illustrate the description of the differential with a pictorial example of the summands of the differential of an element of $\Omega_{\oprd}(\calP^{\dual})(8)$ that corresponds to the  vertex splitting of the vertex $a\in\calP^{\dual}(3)$:
\begin{multline}
\label{eq::d_cobar}
d\left(
\begin{tikzpicture}[scale=0.35]
\node[ext] (v0) at (0,-2) {\tiny{$b$}};
\node[ext] (v1) at (-3,-1) {\tiny{$a$}};
\node (v2) at (2,-1) {\small{$4$}};
%\node[ext] (v3) at (-4,1) {\tiny{$a$}};
\node[ext] (v4) at (-2,1) {\tiny{$c$}};
\node[ext] (v5) at (0,0) {\tiny{$c$}};
%\node[ext] (v6) at (4,1) {\tiny{$b$}};
\coordinate (w0) at (0,-3);
%\node (w1) at (-5,2.5) {\small{$1$}};
\node (w1) at (-5,1) {\small{$1$}};
%\node (w2) at (-3.7,2.5) {\small{$5$}};
\node (w2) at (-3.7,1) {\small{$2$}};
\node (w3) at (-2.7,2.5) {\small{$5$}};
\node (w4) at (-1.3,2.5) {\small{$7$}};
\node (w5) at (-1,1.2) {\small{$3$}};
\node (w6) at (1,1.2) {\small{$6$}};
%\node (w7) at (2,1) {\small{$4$}};
%\node (w8) at (4,1) {\small{$6$}};
%\node (w8) at (3,2.5) {\small{$6$}};
%\node (w9) at (5,2.5) {\small{$9$}};
\draw (v0) edge (w0);
\draw (v1) edge (v0);
%\draw (v3) edge (v1);
\draw (v4) edge (v1);
\draw (v5) edge (v0);
\draw (v2) edge (v0);
%\draw (v6) edge (v2);
%\draw (w1) edge (v3);
\draw (w1) edge (v1);
%\draw (w2) edge (v3);
\draw (w2) edge (v1);
\draw (w3) edge (v4);
\draw (w4) edge (v4);
\draw (w5) edge (v5);
\draw (w6) edge (v5);
%\draw (w7) edge (v2);
%\draw (w8) edge (v6);
%\draw (w8) edge (v2);
%\draw (w9) edge (v6);
\end{tikzpicture}
\right) =  
\underbrace{\pm
\begin{tikzpicture}[scale=0.35]
\node[ext] (v0) at (0,-2) {\tiny{$b$}};
\node[ext] (v1) at (-3,-1) {\tiny{$a'$}};
\node (v2) at (2,-1) {\small{$4$}};
%\node[ext] (v2) at (3,-1) {\tiny{$c$}};
\node[ext] (v3) at (-4,1) {\tiny{$a"$}};
\node[ext] (v4) at (-2,1) {\tiny{$c$}};
\node[ext] (v5) at (0,0) {\tiny{$c$}};
%\node[ext] (v6) at (4,1) {\tiny{$b$}};
\coordinate (w0) at (0,-3);
\node (w1) at (-5,2.5) {\small{$1$}};
%\node (w1) at (-5,1) {\small{$1$}};
\node (w2) at (-3.7,2.5) {\small{$2$}};
%\node (w2) at (-3.7,1) {\small{$2$}};
\node (w3) at (-2.7,2.5) {\small{$5$}};
\node (w4) at (-1.3,2.5) {\small{$7$}};
\node (w5) at (-1,1.2) {\small{$3$}};
\node (w6) at (1,1.2) {\small{$6$}};
%%\node (w7) at (2,1) {\small{$4$}};
%%\node (w8) at (4,1) {\small{$6$}};
%\node (w8) at (3,2.5) {\small{$6$}};
%\node (w9) at (5,2.5) {\small{$9$}};
\draw (v0) edge (w0);
\draw (v1) edge (v0);
\draw (v3) edge (v1);
\draw (v4) edge (v1);
\draw (v5) edge (v0);
\draw (v2) edge (v0);
%\draw (v6) edge (v2);
\draw (w1) edge (v3);
%\draw (w1) edge (v1);
\draw (w2) edge (v3);
%\draw (w2) edge (v1);
\draw (w3) edge (v4);
\draw (w4) edge (v4);
\draw (w5) edge (v5);
\draw (w6) edge (v5);
%%\draw (w7) edge (v2);
%\draw (w8) edge (v6);
%%\draw (w8) edge (v2);
%\draw (w9) edge (v6);
\end{tikzpicture} 
\pm
\begin{tikzpicture}[scale=0.35]
\node[ext] (v0) at (0,-2) {\tiny{$b$}};
\node[ext] (v1) at (-3,-1) {\tiny{$a'$}};
\node (v2) at (2,-1) {\small{$4$}};
\node[ext] (v3) at (-4,1) {\tiny{$a"$}};
\node[ext] (v4) at (-2.5,2.5) {\tiny{$c$}};
\node[ext] (v5) at (0,0) {\tiny{$c$}};
%\node[ext] (v6) at (4,1) {\tiny{$b$}};
\coordinate (w0) at (0,-3);
\node (w1) at (-5,2.5) {\small{$1$}};
%\node (w1) at (-5,1) {\small{$1$}};
%\node (w2) at (-3.7,2.5) {\small{$2$}};
\node (w2) at (-2.5,1) {\small{$2$}};
\node (w3) at (-3.7,3.5) {\small{$5$}};
\node (w4) at (-1.3,3.5) {\small{$7$}};
\node (w5) at (-1,1.2) {\small{$3$}};
\node (w6) at (1,1.2) {\small{$6$}};
%\node (w7) at (2,1) {\small{$4$}};
%\node (w8) at (4,1) {\small{$6$}};
%\node (w8) at (3,2.5) {\small{$6$}};
%\node (w9) at (5,2.5) {\small{$9$}};
\draw (v0) edge (w0);
\draw (v1) edge (v0);
\draw (v3) edge (v1);
\draw (v4) edge (v3);
\draw (v5) edge (v0);
\draw (v2) edge (v0);
%\draw (v6) edge (v2);
\draw (w1) edge (v3);
%\draw (w1) edge (v1);
%\draw (w2) edge (v3);
\draw (w2) edge (v1);
\draw (w3) edge (v4);
\draw (w4) edge (v4);
\draw (w5) edge (v5);
\draw (w6) edge (v5);
%\draw (w7) edge (v2);
%\draw (w8) edge (v6);
%\draw (w8) edge (v2);
%\draw (w9) edge (v6);
\end{tikzpicture} \pm
\begin{tikzpicture}[scale=0.35]
\node[ext] (v0) at (0,-2) {\tiny{$b$}};
\node[ext] (v1) at (-3,-1) {\tiny{$a'$}};
\node (v2) at (2,-1) {\small{$4$}};
\node[ext] (v3) at (-2.7,1) {\tiny{$a"$}};
\node[ext] (v4) at (-2.5,2.5) {\tiny{$c$}};
\node[ext] (v5) at (0,0) {\tiny{$c$}};
%\node[ext] (v6) at (4,1) {\tiny{$b$}};
\coordinate (w0) at (0,-3);
%\node (w1) at (-5,2.5) {\small{$1$}};
\node (w1) at (-5,1) {\small{$1$}};
\node (w2) at (-3.7,2.5) {\small{$2$}};
%\node (w2) at (-2.5,1) {\small{$2$}};
\node (w3) at (-3.7,3.5) {\small{$5$}};
\node (w4) at (-1.3,3.5) {\small{$7$}};
\node (w5) at (-1,1.2) {\small{$3$}};
\node (w6) at (1,1.2) {\small{$6$}};
%\node (w7) at (2,1) {\small{$4$}};
%\node (w8) at (4,1) {\small{$6$}};
%\node (w8) at (3,2.5) {\small{$6$}};
%\node (w9) at (5,2.5) {\small{$9$}};
\draw (v0) edge (w0);
\draw (v1) edge (v0);
\draw (v3) edge (v1);
\draw (v4) edge (v3);
\draw (v5) edge (v0);
\draw (v2) edge (v0);
%\draw (v6) edge (v2);
%\draw (w1) edge (v3);
\draw (w1) edge (v1);
\draw (w2) edge (v3);
%\draw (w2) edge (v1);
\draw (w3) edge (v4);
\draw (w4) edge (v4);
\draw (w5) edge (v5);
\draw (w6) edge (v5);
%\draw (w7) edge (v2);
%\draw (w8) edge (v6);
%\draw (w8) edge (v2);
%\draw (w9) edge (v6);
\end{tikzpicture}
}_{\text{vertex splitting of ${a}$}} +\ldots
\end{multline}

The bar and cobar construction functors were initially defined for the category of associative (co)algebras and nowdays are well understood in many other monoidal categories.
We denote by $\Omega_{\perm}(\ttt)$ and $\calB_{\perm}(\ttt)$ the cobar and bar constructions for  twisted associative algebras and denote by $\Omega_{1\ttt2\oprd}(\ttt)$/$\calB_{1\ttt2\oprd}(\ttt)$ -- the cobar/bar construction functors for the category of $1\ttt2$-operads.
The elements of the dg $1\ttt2$-operad $\Omega_{1\ttt2\oprd}(\calP^{\dual})$ are given by operadic trees whose edges are colored in two different colors. However, there exists exactly one input of the second (dotted) color. The following pictorial example of the differential of an element in $\Omega_{1\ttt2\oprd}(\calP^{\dual})(7,1)^{2}$ shows the relationship of the operadic and $1\ttt2$-operadic (co)bar constructions:
\begin{multline}
\label{eq::d_cobar_12}
d\left(
\begin{tikzpicture}[scale=0.35]
\node[ext] (v0) at (0,-2) {\tiny{$b$}};
\node[ext] (v1) at (-3,-1) {\tiny{$a$}};
\node (v2) at (2,-1) {\small{$3$}};
%\node[ext] (v3) at (-4,1) {\tiny{$a$}};
\node[ext] (v4) at (-2,1) {\tiny{$c$}};
\node[ext] (v5) at (0,0) {\tiny{$c$}};
%\node[ext] (v6) at (4,1) {\tiny{$b$}};
\coordinate (w0) at (0,-3);
%\node (w1) at (-5,2.5) {\small{$1$}};
\node (w1) at (-5,1) {\small{$\bar{1}$}};
%\node (w2) at (-3.7,2.5) {\small{$5$}};
\node (w2) at (-3.7,1) {\small{$1$}};
\node (w3) at (-2.7,2.5) {\small{$4$}};
\node (w4) at (-1.3,2.5) {\small{$6$}};
\node (w5) at (-1,1.2) {\small{$2$}};
\node (w6) at (1,1.2) {\small{$5$}};
%\node (w7) at (2,1) {\small{$3$}};
%\node (w8) at (4,1) {\small{$5$}};
%\node (w8) at (3,2.5) {\small{$6$}};
%\node (w9) at (5,2.5) {\small{$9$}};
\draw[dotted] (v0) edge (w0);
\draw[dotted] (v1) edge (v0);
%\draw (v3) edge (v1);
\draw (v4) edge (v1);
\draw (v5) edge (v0);
\draw (v2) edge (v0);
%\draw (v6) edge (v2);
%\draw (w1) edge (v3);
\draw[dotted] (w1) edge (v1);
%\draw (w2) edge (v3);
\draw (w2) edge (v1);
\draw (w3) edge (v4);
\draw (w4) edge (v4);
\draw (w5) edge (v5);
\draw (w6) edge (v5);
%\draw (w7) edge (v2);
%\draw (w8) edge (v6);
%\draw (w8) edge (v2);
%\draw (w9) edge (v6);
\end{tikzpicture}
\right) =  
\underbrace{
	\begin{tikzpicture}[scale=0.35]
	\node[ext] (v0) at (0,-2) {\tiny{$b$}};
	\node[ext] (v1) at (-3,-1) {\tiny{$a'$}};
	\node (v2) at (2,-1) {\small{$3$}};
	\node[ext] (v3) at (-4,1) {\tiny{$a"$}};
	\node[ext] (v4) at (-2,1) {\tiny{$c$}};
	\node[ext] (v5) at (0,0) {\tiny{$c$}};
	%\node[ext] (v6) at (4,1) {\tiny{$b$}};
	\coordinate (w0) at (0,-3);
	\node (w1) at (-5,2.5) {\small{$\bar{1}$}};
	%\node (w1) at (-5,1) {\small{$1$}};
	\node (w2) at (-3.7,2.5) {\small{$1$}};
	%\node (w2) at (-3.7,1) {\small{$2$}};
	\node (w3) at (-2.7,2.5) {\small{$4$}};
	\node (w4) at (-1.3,2.5) {\small{$6$}};
	\node (w5) at (-1,1.2) {\small{$2$}};
	\node (w6) at (1,1.2) {\small{$5$}};
%	\node (w7) at (2,1) {\small{$3$}};
%	\node (w8) at (4,1) {\small{$5$}};
	%\node (w8) at (3,2.5) {\small{$6$}};
	%\node (w9) at (5,2.5) {\small{$9$}};
	\draw[dotted] (v0) edge (w0);
	\draw[dotted] (v1) edge (v0);
	\draw[dotted] (v3) edge (v1);
	\draw (v4) edge (v1);
	\draw (v5) edge (v0);
	\draw (v2) edge (v0);
	%\draw (v6) edge (v2);
	\draw[dotted] (w1) edge (v3);
	%\draw (w1) edge (v1);
	\draw (w2) edge (v3);
	%\draw (w2) edge (v1);
	\draw (w3) edge (v4);
	\draw (w4) edge (v4);
	\draw (w5) edge (v5);
	\draw (w6) edge (v5);
%	\draw (w7) edge (v2);
	%\draw (w8) edge (v6);
%	\draw (w8) edge (v2);
	%\draw (w9) edge (v6);
	\end{tikzpicture} 
	\pm
	\begin{tikzpicture}[scale=0.35]
	\node[ext] (v0) at (0,-2) {\tiny{$b$}};
	\node[ext] (v1) at (-3,-1) {\tiny{$a'$}};
	\node (v2) at (2,-1) {\small{$3$}};
	\node[ext] (v3) at (-4,1) {\tiny{$a"$}};
	\node[ext] (v4) at (-2.5,2.5) {\tiny{$c$}};
	\node[ext] (v5) at (0,0) {\tiny{$c$}};
	%\node[ext] (v6) at (4,1) {\tiny{$b$}};
	\coordinate (w0) at (0,-3);
	\node (w1) at (-5,2.5) {\small{$\bar{1}$}};
	%\node (w1) at (-5,1) {\small{$1$}};
	%\node (w2) at (-3.7,2.5) {\small{$2$}};
	\node (w2) at (-2.5,1) {\small{$1$}};
	\node (w3) at (-3.7,3.5) {\small{$4$}};
	\node (w4) at (-1.3,3.5) {\small{$6$}};
	\node (w5) at (-1,1.2) {\small{$2$}};
	\node (w6) at (1,1.2) {\small{$5$}};
%	\node (w7) at (2,1) {\small{$3$}};
%	\node (w8) at (4,1) {\small{$5$}};
	%\node (w8) at (3,2.5) {\small{$6$}};
	%\node (w9) at (5,2.5) {\small{$9$}};
	\draw[dotted] (v0) edge (w0);
	\draw[dotted] (v1) edge (v0);
	\draw[dotted] (v3) edge (v1);
	\draw (v4) edge (v3);
	\draw (v5) edge (v0);
	\draw (v2) edge (v0);
	%\draw (v6) edge (v2);
	\draw[dotted] (w1) edge (v3);
	%\draw (w1) edge (v1);
	%\draw (w2) edge (v3);
	\draw (w2) edge (v1);
	\draw (w3) edge (v4);
	\draw (w4) edge (v4);
	\draw (w5) edge (v5);
	\draw (w6) edge (v5);
%	\draw (w7) edge (v2);
	%\draw (w8) edge (v6);
%	\draw (w8) edge (v2);
	%\draw (w9) edge (v6);
	\end{tikzpicture} \pm
	\begin{tikzpicture}[scale=0.35]
	\node[ext] (v0) at (0,-2) {\tiny{$b$}};
	\node[ext] (v1) at (-3,-1) {\tiny{$a'$}};
	\node (v2) at (2,-1) {\small{$3$}};
	\node[ext] (v3) at (-2.7,1) {\tiny{$a"$}};
	\node[ext] (v4) at (-2.5,2.5) {\tiny{$c$}};
	\node[ext] (v5) at (0,0) {\tiny{$c$}};
	%\node[ext] (v6) at (4,1) {\tiny{$b$}};
	\coordinate (w0) at (0,-3);
	%\node (w1) at (-5,2.5) {\small{$1$}};
	\node (w1) at (-5,1) {\small{$\bar{1}$}};
	\node (w2) at (-3.7,2.5) {\small{$1$}};
	%\node (w2) at (-2.5,1) {\small{$2$}};
	\node (w3) at (-3.7,3.5) {\small{$4$}};
	\node (w4) at (-1.3,3.5) {\small{$6$}};
	\node (w5) at (-1,1.2) {\small{$2$}};
	\node (w6) at (1,1.2) {\small{$5$}};
%	\node (w7) at (2,1) {\small{$3$}};
%	\node (w8) at (4,1) {\small{$5$}};
	%\node (w8) at (3,2.5) {\small{$6$}};
	%\node (w9) at (5,2.5) {\small{$9$}};
	\draw[dotted] (v0) edge (w0);
	\draw[dotted] (v1) edge (v0);
	\draw (v3) edge (v1);
	\draw (v4) edge (v3);
	\draw (v5) edge (v0);
	\draw (v2) edge (v0);
	%\draw (v6) edge (v2);
	%\draw (w1) edge (v3);
	\draw[dotted] (w1) edge (v1);
	\draw (w2) edge (v3);
	%\draw (w2) edge (v1);
	\draw (w3) edge (v4);
	\draw (w4) edge (v4);
	\draw (w5) edge (v5);
	\draw (w6) edge (v5);
%	\draw (w7) edge (v2);
	%\draw (w8) edge (v6);
%	\draw (w8) edge (v2);
	%\draw (w9) edge (v6);
	\end{tikzpicture}
}_{\text{vertex splitting of ${a}$}} +\ldots
\end{multline}

\begin{proposition}
\label{prp::+Koszul}	
 The functor ${+}:\calP\mapsto \calP_{+}$ commutes with Koszul duality, meaning that there exists the canonical (functorial in $\calP$) isomorphism:
 \begin{equation}
 \label{eq::bar::constr}
 \left(\Omega_{\oprd}(\calP^{\dual})\right)_{+} \cong \Omega_{1\ttt2\oprd}((\calP_{+})^{\dual})
 \end{equation}
% Here $\Omega_{\oprd}(\ttt)$ is the cobar construction of a symmetric operad and $\Omega_{1\ttt2\oprd}(\ttt)$ denotes the cobar construction of an operad with two colors.
\end{proposition}
\begin{proof}	
The cobar construction $\Omega_{\oprd}(\calP^{\dual})$ is the free symmetric operad generated by an $\bbS$-collection $\calP^{\dual}$. Thanks to Corollary~\ref{cor::+::free} we know that the left hand side of~\eqref{eq::bar::constr} is the free $1\ttt2$-operad generated by $\calP^{\dual}$ and $\partial\calP^{\dual}$ and hence is isomorphic to $\Omega_{1\ttt2\oprd}(\calP_{+}^{\dual})$ as a $1\ttt2$-operad. The coincidence of differentials is obvious from pictorial descriptions~\eqref{eq::d_cobar} and \eqref{eq::d_cobar_12}. 
\end{proof}

\subsection{$1\ttt2$-cobar-bar resolution of $\dcalP$ as a right $\calP$-module}
\label{sec::perm::Bar::Cobar}
The $1\ttt2$-operadic bar and cobar constructions and Proposition~\ref{prp::+Koszul} lead to the following Theorem, which is very useful for applications after Definition~\ref{prp::perm::UP} of the derived associative universal enveloping algebra.
\begin{theorem}
\label{thm::koszul::permutad}	
There exists a free resolution of $\dcalP$ in the category of right $\calP$-modules generated by 
the cobar construction of the twisted associative algebra $\partial(\calB_{\oprd}(\calP))$: 
\begin{equation}
\label{eq::resol}
\left(\Omega_{\perm}(\partial(\calB_{\oprd}(\calP))\circ \calP,d_{\calB}+d_{\Omega}+d_{\calP}\right) \stackrel{quis}{\rightarrow} \dcalP %\calP_{+}^{2}.
\end{equation}
where the total differential consists of three summands: 
\begin{itemize}
	\setlength{\itemsep}{-0.2em}
	\item $d_{\calB}$ is the differential in the bar construction $(\calB_{\oprd}(\ttt),d_{\calB})$;
	\item $d_{\Omega}$ is the differential in the cobar construction $(\Omega_{\perm}(\ttt),d_{\Omega})$;
	\item $d_{\calP}$ is the summand of the differential in the free resolution $(\calB_{\oprd}(\calP)\circ \calP, d_{\calB} + d_{\calP}) \twoheadrightarrow \kk$ that corresponds to the interaction of $\calP$ and the bar construction of $\calP$.
\end{itemize}
\end{theorem}
\begin{proof}
The bar-cobar construction~\eqref{eq::bar::cobar} for operads predicts the following quasi-isomorphism for any symmetric operad:
\begin{equation}
\label{eq::cobar::bar}
\pi:\Omega_{\oprd}(\calB_{\oprd}(\calP)) \twoheadrightarrow \calP.
\end{equation}	
Thanks to Proposition~\ref{prp::+Koszul} we have the following (quasi)-isomorphisms
\[
\Omega_{1\ttt2\oprd}\left(\calB_{1\ttt2\oprd}(\calP_{+})\right) =
\Omega_{1\ttt2\oprd}\left((\calB_{\oprd}(\calP))_{+}\right) = \left(\Omega_{\oprd}(\calB_{\oprd}(\calP))\right)_{+}  \twoheadrightarrow \calP_{+}.
\]
The elements of the cobar construction $\Omega_{1\ttt2\oprd}\left(\calB_{1\ttt2\oprd}(\calP_{+})\right)$ are given by operadic trees whose vertices are labeled by elements of the quasi-free $1\ttt2$-cooperad  
$\calB_{1\ttt2\oprd}(\calP_{+})$. In particular, the maximal subtrees without vertices $\partial \calB_{\oprd}(\calP)$ with the output of the second dotted color can be identified with the elements of $\Omega_{\oprd}(\calB_{\oprd}(\calP))$.
The remaining vertices of the operadic tree belong to the twisted associative algebra $\Omega_{\perm}(\partial\calB_{\oprd}(\calP))$.
 Thus, the surjective quasi-isomorphism~\eqref{eq::cobar::bar} leads to the following quasi-isomorphic surjection of the space of operations with the output of the second color:
\begin{multline*}
(\Omega_{1\ttt2\oprd}\left(\calB_{1\ttt2\oprd}(\calP_{+})\right)^{2}, d_{\Omega}+d_{\calB}) =  \\
= (\Omega_{\perm}\left(\partial(\calB_{\oprd}(\calP))\right)\circ \Omega_{\oprd}(\calB_{\oprd}(\calP)),d_{\Omega}+d_{\calB}) \stackrel{\pi}{\longrightarrow} 
(\Omega_{\perm}\left(\partial(\calB_{\oprd}(\calP))\right)\circ\calP,d_{\Omega}+d_{\calB}+d_{\calP})
\end{multline*}
Here we identify the cobar constructions $\Omega_{1\ttt2\oprd}(\partial(\ttt))$ and $\Omega_{\perm}(\partial(\ttt))$ as in Proposition~\ref{st::perm2oper}.
\end{proof}

\begin{corollary}
\label{cor::Koszul:resol}	
 If the symmetric operad $\calP$ is Koszul then there exists a smaller resolution of $\partial\calP$ as a right $\calP$-module:
 \begin{equation}
 \label{eq::kosz::resol}
 \left(\Omega_{\perm}(\partial\calP^{\antish})\circ\calP,d_{\Omega}+d_K\right) \stackrel{quis}{\twoheadrightarrow} \dcalP 
 \end{equation}
where $d_K: \calP^{!}\circ\calP \to  \calP^{!}\circ\calP$ is the standard Koszul differential.
\end{corollary}
\begin{proof}
The Koszulness of the symmetric operad $\calP$ implies (and even is equivalent) to the following statement: "The surjection $\calB_{\oprd}(\calP) \twoheadrightarrow \calP^{\antish}$ is a quasi-isomorphism." 
Consequently, the bar-cobar resolution~\eqref{eq::resol} is equivalent to the resolution~\eqref{eq::kosz::resol}. 
\end{proof}

\section{Necessary and sufficient conditions on the PBW property of $\UP$}
\label{sec::PBW::UP}

\subsection{Main criterion for the PBW property of $\UU_{\calP}$ for a Koszul operad $\calP$} 
\label{sec::main::PBW}

Thanks to Corollary~\ref{cor::Koszul:resol} we know that if the Koszul operad $\calP$ is an ordinary operad (meaning that there is no additional homological grading and a differential) then $\UP$ satisfies PBW if and only if the homology of the complex $\Omega_{\perm}(\dcalP^{\antish})$ differs from zero only in $0$-homological degree.

\begin{theorem}
\label{cor::Perm::UP::Koszul}
The universal enveloping functor $\UP$ of a symmetric Koszul operad $\calP$ generated by a symmetric collection $\Upsilon$ satisfies PBW if and only if the twisted associative algebra $\dcalP^{!}$ assigned to a Koszul dual operad $\calP^{!}$ is a 
Koszul twisted associative algebra generated by the derivative $\partial \Upsilon^{\dual}$ of the set of operadic generators $\Upsilon^{\dual}$ of $\calP^{!}$. 
\end{theorem}	
\begin{proof}
Suppose that $\dcalP^{!}$ is a quadratic and  Koszul twisted associative algebra. Then we have the canonical quasiisomorphism:
$$(\Omega_{\perm}(\partial(\calP_{\oprd}^{\antish})), d_{\Omega} )\rightarrow (\partial(\calP_{\oprd}^{!}))^{!}_{\perm}$$
The restriction on the set of generators of $\dcalP^{!}$ implies that the spectral sequence associated with the bicomplex~\eqref{eq::kosz::resol} degenerates in the first term for degree reasons and we have an isomorphism of the cohomology  $\dcalP\simeq(\partial(\calP_{\oprd}^{!}))^{!}_{\perm}\circ\calP$.

In order to prove the \emph{if} statement we have to observe several facts. First, since the operad $\calP$ is quadratic the universal enveloping algebra $\UP(V)$ is an associative algebra generated by $\partial\Upsilon(V)$ subject to quadratic-linear relations (Proposition~\ref{prp::+Ex}). Therefore, if $\UP$ satisfies PBW then the twisted associative algebra $\UPZ$ is quadratic and is generated by $\partial\Upsilon$. Second, thanks to Corollary~\ref{cor::PBW::free::P} we know that $\calP_{+}$ is a free right $\calP$-module generated by $\UPZ$. Third, since $\calP$ is Koszul, Corollary~\ref{cor::Koszul:resol} predicts that $\UPZ$ has to be equivalent (quasiisomorphic) to  $\Omega_{\perm}((\calP_{\oprd}^{\antish})_{+}^{2})$.
The latter means that the quadratic twisted associative algebra $\UPZ$ and the twisted associative algebra $(\calP_{\oprd}^{!})_{+}^{2}$ are bar-dual to each other. It follows that the quadratic twisted associative algebra $\UPZ$ is Koszul and $\partial(\calP_{\oprd}^{!})$ is the Koszul dual and consequently a quadratic twisted associative algebra.
\end{proof}
Suppose that we are under the assumptions of Theorem~\ref{cor::Perm::UP::Koszul}: the $\calP$ is a quadratic Koszul operad and the corresponding twisted associative algebra $\dcalP^{!}_{\oprd}$ is also quadratic Koszul with the same set of generators. Consider the operadic Koszul complex $K^{\udot}(\calP):=(s\calP^{\antish}_{\oprd} \circ \calP,d)$ that has to be acyclic thanks to the Koszul property. The Koszul differential $d$ is a composition of one cocomposition and one composition which can be written for binary operads in the following way:
\begin{multline}
\label{eq::Koszul::complex::operad}
\calP^{\antish}_{\oprd}(n)\otimes (\calP(m_1)\otimes\ldots\otimes\calP(m_n)) \mapsto \oplus (\oplus_{ij}\calP^{\antish}_{\oprd}(n-1)\circ_i \calP(2) ) \otimes (\calP(m_1)\otimes\ldots\otimes\calP(m_n)) =
\\
=\oplus_{ij}\calP^{\antish}_{\oprd}(n-1)  \otimes (\ldots\otimes(\calP(2)\circ\calP(m_i)\otimes\calP(m_j))\otimes\ldots) \mapsto \oplus_{ij}\calP^{\antish}_{\oprd}(n-1) \otimes ( \ldots\otimes\calP(m_i+m_j)\otimes\ldots) 
\end{multline}
Let us now take the derivative $\partial K^{\udot}(\calP)$ of the symmetric collection of the aforementioned Koszul complexes $K^{\udot}(\calP)$.
In other words, let us mark the first input in all these complexes. (In our pictures we make the path from the first input to the output dotted; the composition through the corresponding color is also denoted via dotted circle.) The corresponding differential will be splitten in two differentials: $d=d_{\perm}+d_{\oprd}$, where $d_{\perm}$ consists of those summands that interacts with the dotted path. 
 The tensor product over $\calP$ of the derivative of the Koszul complex $\partial K^{\udot}(\calP)$ (considered as a free right $\calP$-module) and a trivial $\calP$-algebra $\uk$ kills the second part of the differential $d_{\oprd}$ and  what remains coincides with the Koszul complex for the twisted associative coalgebra $\dcalP^{\antish}$:
 \[
 \partial K^{\udot}(\calP)\circ_{\calP}\uk \simeq \left(\partial (s\calP^{\antish}_{\oprd}) \dcirc (\dcalP\circ_{\calP} \uk),d_{\perm}\right) \simeq \left(s\partial(\calP^{\antish}_{\oprd})\dcirc \UPZ, d_{\perm} \right)
 \]
The latter is known to be acyclic by the assumption on Koszulness of the twisted associative algebra $\partial(\calP^{\antish}_{\oprd})$.
While taking the tensor product of $\partial K^{\udot}(\calP)$ over $\calP$ with a $\calP$-algebra $V$ we end up with the complex of left $\UP(V)$-modules (we call it the Koszul complex):
\begin{equation}
\label{eq::UP::resolution}
K^{\udot}_{\UP}(V):= \left(s\partial(\calP^{\antish}_{\oprd})(V)\dcirc \UP(V), d_{\perm}+d_{\oprd} \right)
\end{equation}
\begin{corollary}
	\label{cor::UP::Koszul}	
	If $\calP$ is a Koszul operad such that $\UP$ satisfies PBW, then
\begin{enumerate}
\setlength{\itemsep}{-0.4em}
\item	There is an isomorphism of twisted associative algebras:
	\[ \UPZ \simeq (\partial(\calP^{!}_{\oprd}))^{!}_{\perm} \]
\item For any $\calP$-algebra $V$ its universal enveloping $\UP(V)$ is a quadratic-linear Koszul algebra, whose associated graded Koszul algebra is isomorphic to $\UP(V_0)$ and 
\[
\UP(V_0)^{\antish}_{\alg} \simeq (\partial\left( s(\calP^{\antish}_{\oprd})\right))(V):= \oplus_{n\geq 0} \calP^{\antish}_{\oprd}(n+1)\otimes_{\bbS_n}V^{\otimes n}[n]\otimes \Sgn_n;
\]
\item
	In particular, the Koszul complex $K^{\udot}_{\UP}(V)$ 
	defines a free resolution of $\kk$ in the category of left $\UP(V)$-modules.
\end{enumerate} 
\end{corollary}
\begin{proof}
We already proved the first item while working out the proof of Theorem~\ref{cor::Perm::UP::Koszul}.
The Schur-Weyl duality theorem translates the statement for twisted associative algebras in terms of the corresponding statement for the trivial $\calP$-algebra $V_0$ given in the second item.
The case of nontrivial $\calP$-algebras is covered by general theory of nonhomogeneous Koszul duality treated in~\cite{PP} chapter 5.
\end{proof}

\begin{example}
The operads $\Com$ and $\Lie$ of commutative and Lie algebras are known to be Koszul dual to each other. The Koszulness of the twisted associative algebra $\partial\Com$ was established in~\cite{DK::Shuffle::Pattern}, where we found a quadratic Gr\"obner basis for the corresponding shuffle algebra.  
The twisted associative algebra $\partial\Lie$ is free (and hence Koszul). Moreover, $\partial\Lie$ is  generated by a single one-dimensional $\bbS_1$-representation, since 
\[
Res_{\bbS_n}^{\bbS_{n+1}}\Lie(n+1) \simeq Res_{\bbS_n}^{\bbS_{n+1}} Ind_{\ZZ_{n+1}}^{S_{n+1}}\mathbb{C}_{\sqrt[n+1]{1}} \simeq \kk[\bbS_n]
\]
Here we denote by $\mathbb{C}_{\sqrt[n+1]{1}}$ the one-dimensional representation of $\ZZ_{n+1}$ where the generator of the cyclic group acts by a primitive root of unity of order $(n+1)$.
Thus, Theorem~\ref{cor::Perm::UP::Koszul} reproves the PBW property for $\UU_{\Lie}$ and $\UU_{\Com}$, and Corollary~\ref{cor::UP::Koszul} reproduces the well-known description of the corresponding universal enveloping algebras:
\[
\gr^{PBW}\UU_{\Lie}(\frg) \simeq S(\frg), \quad \UU_{\Com}(A) = \kk\oplus A.
\]
The corresponding Koszul complexes coincide with the Chevalley-Eilenberg resolution and with the bar-resolution respectively:
\[
(U(\frg)\otimes \Lambda^{\udot}(\frg[1]),d) \rightarrow \kk, \ \quad
(A\otimes T^{\udot}(A[1]),d) \rightarrow \kk
\]
\end{example}

\begin{example}
\label{ex::Alg2Op}	
Recall that to each commutative associative graded algebra  $A:=\oplus_{n\geq 0} A_n$ one can assign a symmetric operad $\calO_{A}$ whose space of $n$-ary operations $\calO_{A}(n)$ is isomorphic to $A_{n-1}$ with the trivial $\bbS_n$ action and the composition rules are given by multiplication:
\[
\circ_i: \calO_{A}(m) \otimes \calO_{A}(n) = A_{m-1}\otimes A_{n-1} \rightarrow A_{m+n-2} = \calO_{A}(m+n-1).
\]
Moreover, it was proved in~\cite{DK::Anick}[Thm.~5.3] that if the given algebra $A$ is  Koszul then the corresponding operad $\calO_{A}$ is also Koszul.
The straightforward generalization to the case of colored operads (and in particular, for permutads) of the proof suggested in~\cite{DK::Anick} shows that koszulness of $A$ also implies the  koszulness of the permutad $\partial\calO_{A}$.
Consequently, thanks to Theorem~\ref{cor::Perm::UP::Koszul} the universal enveloping functor $\UU_{{(\calO_{A})}^{!}_{\oprd}}$ satisfies the PBW property, and for any $\calO_{A}^{!}$-algebra $V$ the associative algebra  $\UU_{{(\calO_{A})}^{!}_{\oprd}}(V)$ is a nonhomogeneous Koszul algebra whose Koszul-dual is isomorphic to the Hadamard product of graded quadratic Koszul algebras $A$ and $\Lambda^{\udot}(V)$:
\[
\left(\gr^{PBW}\UU_{{(\calO_{A})}^{!}_{\oprd}}(V)\right)_{\alg}^{!} \simeq  \bigoplus_{n\geq 0} A_n \otimes \Lambda^n V =: A\boxtimes \Lambda^{\udot}V
\] 
\end{example}

\subsection{Hilbert series and a necessary condition for PBW of $\UP$} 
\label{sec::PBW::Hilb}

A necessary condition on the operad $\calP$ to have the PBW property is formulated in terms of the Hilbert series of dimensions (characters) of this operad. 
Recall that to each $\bbS$-collection $\calP$  one can assign the two standard generating series:
\begin{gather*}
f_\calP(t):= \sum_{n\geq 1}\frac{\dim{\calP(n)}}{n!} t^{n}; \\
\rchi_{\calP}(p_1,p_2,\ldots)= \sum_{n\geq 1} \rchi_{\bbS_n}(\calP(n)), \end{gather*}
Here  $\rchi_{\bbS_n}(V):= \sum_{\rho\vdash n} \frac{p_{\rho}}{z_\rho}\mathsf{Tr}_V(\rho)$ is a symmetric function of degree $n$ associated with the corresponding $\bbS_n$-character of the symmetric group given in the basis of Newton's sums $p_{k}:=\sum x_i^{k}$. 
In particular, $f_{\calP}(t) = \rchi_{\calP}(t,0,0,\ldots)$.
Recall that the generating series of characters of a given Koszul operad $\calP$ and its Koszul dual $\calP^{!}$ are almost inverse to each other (e.g.~\cite{Ginz_Kapranov,DK::Lie2}):
\[ \epsilon(\rchi_{\calP}) \circ \epsilon(\rchi_{\calP^{!}}) 
%=  \rchi_{\calP}(-p_1,-p_2,\ldots) \circ \rchi_{\calP^{!}}(-p_1,-p_2,\ldots) 
= p_1
\]
Here the automorphism $\epsilon:{\mathbb{\Lambda}} \to {\mathbb{\Lambda}}$ of the ring of symmetric functions ${\mathbb{\Lambda}}:=\mathbb{Z}[x_1,x_2,\ldots]^{\bbS}$ sends $p_i\to -p_i$. On the level of $\bbS$-representations $\epsilon$ changes the parity of a representation and tensor it with a sign representation.
\begin{theorem}
\label{thm:Hilb::ser::PBW}
The generating series of the $\bbS$-collection $\UPZ$ assigned to the universal enveloping functor $\UP$ of a symmetric operad $\calP$ that satisfies the PBW property has the following description:
\begin{gather}
\label{eq::UPZ::character}
f_{\UPZ}(t)= -\left(\frac{\partial f_{\calP^{!}}(-t)}{\partial t}\right)^{-1},
\quad %\text{ with } f_{\calP}(-t)\circ f_{\calP^{!}}(-t) = f_{\calP^{!}}(-t)\circ f_{\calP}(-t)=t,\\
\rchi_{\UPZ} = - \left(\frac{\partial}{\partial p_1} \rchi_{\calP^{!}}(-p_1,-p_2,\ldots)\right)^{-1}
\end{gather}
\end{theorem}
\begin{proof}
It is enough to prove~\eqref{eq::UPZ::character} for symmetric functions, since the first equality on Hilbert series is the substitution: $p_1=t$, $p_k=0$ for $k>1$.
The PBW property gives the following relation:
	\begin{equation}
	\label{eq::char::P+}
	\rchi_{\UPZ}\circ \rchi_{\calP} = \rchi_{\dcalP} = \frac{\partial}{\partial p_1} \rchi_{\calP} 
	\end{equation}
	Consequently, the Equation~\eqref{eq::char::P+} is equivalent to the following
	\begin{multline*}
	\rchi_{\UPZ} = \rchi_{\UPZ}\circ \epsilon(\rchi_{\calP}) \circ \epsilon(\rchi_{\calP^!}) = \epsilon(\rchi_{\UPZ}\circ \rchi_{\calP}) \circ \epsilon(\rchi_{\calP^!}) = 
	\epsilon\left(\frac{\partial}{\partial p_1} \rchi_{\calP} \right) \circ \epsilon(\rchi_{\calP^!}) = 
	\\
	=
	-\left(\frac{\partial}{\partial p_1} \epsilon(\rchi_{\calP}) \right) \circ \epsilon(\rchi_{\calP^!}) = 
	-	\frac{ \frac{\partial}{\partial p_1} \left( \epsilon(\rchi_{\calP}) \circ \epsilon(\rchi_{\calP^!}) \right) }{\frac{\partial}{\partial p_1} \epsilon(\rchi_{\calP^!})  } = - \frac{ \frac{\partial}{\partial p_1} p_1 }  {\frac{\partial}{\partial p_1} \epsilon(\rchi_{\calP^!})  } =
	- \left( \frac{\partial}{\partial p_1} \epsilon(\rchi_{\calP^!})\right)^{-1}
	\end{multline*}
\end{proof}

\begin{corollary}
The generating series $-\left(\frac{\partial f_{\calP^{!}}(-t)}{\partial t}\right)^{-1}$ of a symmetric (Koszul) operad $\calP$ whose universal enveloping functor yields the PBW property has only nonnegative coefficients. Moreover, the symmetric function $-\left(\frac{\partial \epsilon(\rchi_{\calP^{!}})}{\partial p_1}\right)^{-1}$ has to be Schur-positive.
\end{corollary}

\begin{example}
\label{ex::Poisson}	
The universal enveloping functor of the Poisson operad $\Pois$ does not satisfy PBW.
\end{example}
\begin{proof}	
Consider the grading of the Poisson operad by the number of Lie brackets.  We can write the corresponding generating series that depends on an additional parameter $q$ that counts the number of Lie brackets:
\[
f_{\Pois}(t,q) = f_{\Com}(t)\circ \frac{f_{\Lie}(qt)}{q} = (e^{t}-1)\circ\left( -\frac{\ln(1-qt)}{q}\right).
\]	 
The Poisson operad is known to be Koszul self dual and consequently we have:
\[
-\left(\frac{\partial f_{\Pois^{!}}(-t,q)}{\partial t}\right)^{-1} = 
(1+qt)^{1+q^{-1}} = 1 + (1+q) t + \frac{(1+q)}{2}t^{2} + \frac{q^2-1}{6} t^{3} + \ldots 
\]
whose coefficient of $t^3$ is already a polynomial in $q$ with a negative coefficient. Thus $\UU_{\Pois}$ does not satisfy PBW.
\end{proof}

\subsection{Gr\"obner bases and a sufficient condition for PBW of $\UP$}
\label{sec::PBW::Grob}

Theorem~\ref{cor::Perm::UP::Koszul} defines a nice sufficient condition for the $PBW$ property in terms of the Koszul property for appropriate operads and twisted associative algebras. 
At the moment the most effective way to check the Koszul property is to find a quadratic Gr\"obner basis.
In this section we state a useful sufficient condition for the PBW property in terms of Gr\"obner bases for operads and show how it works for particular examples in the subsequent Section~\S\ref{sec::examples}. Let us briefly recall the main definition of~\cite{DK::Grob} motivated by PBW theory suggested in~\cite{Hoffbeck}.

\begin{definition}
A shuffle operad $\calP$ in the category of vector spaces consist of a collection of vector spaces $\calP(n)$ and composition rules
\begin{equation}
\label{eq::Shufle:comp}
\circ_{\sigma}: \calP(k)\otimes \calP(n_1)\otimes\ldots\otimes\calP(n_k) \rightarrow \calP(n_1+\ldots+n_k)
\end{equation}
 numbered by surjections $\sigma: [1 (n_1+\ldots +n_k)] \twoheadrightarrow [1 k]$ of finite ordered sets yielding the shuffle condition:
\begin{equation}
\label{eq::shuffle}
 1\leq i< j \leq k \Rightarrow \min (\sigma^{-1}(i)) < \min (\sigma^{-1}(j)).
\end{equation}
The composition rule~\eqref{eq::Shufle:comp} has to satisfy  the appropriate associativity condition that can be roughly written in the following form:
\[
\calP \circ_{\sigma} (\calP \circ_{\tau} \calP) = (\calP \circ_{\sigma} \calP) \circ_{\tau} \calP.
\]
\end{definition}
With any symmetric operad $\calP$ one can assign a shuffle operad $\shfl(\calP)$ by forgetting the symmetric group action and certain part of the compositions.  The forgetful functor $\shfl$ does not change the underlying collection of vector spaces and maps free symmetric operads to free shuffle operads. Therefore, the koszulness of the shuffle operad $\shfl(\calP)$ implies the koszulness of the underlying symmetric operad $\calP$.
The main advantage of shuffle operads is that there is a notion of monomials in the free shuffle operad. Each monomial can be uniquely presented as a planar tree whose vertices are labeled by generators and leaves are labeled such that for each inner vertex $v$ the ordering of the minims of the leaves of subtrees of $v$ respects the planar ordering. (See example~\eqref{pic::shuffle} below of a shuffle monomial of arity $9$ in the free shuffle operad on $4$ generators of arities $2$ and $3$.)
\begin{equation}
\label{pic::shuffle}
\begin{tikzpicture}[scale=0.35]
\node[ext] (v0) at (0,-2) {\tiny{$d$}};
\node[ext] (v1) at (-3,-1) {\tiny{$a$}};
\node[ext] (v2) at (3,-1) {\tiny{$c$}};
\node[ext] (v3) at (-4,1) {\tiny{$a$}};
\node[ext] (v4) at (-2,1) {\tiny{$c$}};
\node[ext] (v5) at (0,0) {\tiny{$c$}};
\node[ext] (v6) at (4,1) {\tiny{$b$}};
\coordinate (w0) at (0,-3);
\node (w1) at (-5,2.5) {\small{$1$}};
\node (w2) at (-3.7,2.5) {\small{$5$}};
\node (w3) at (-2.7,2.5) {\small{$2$}};
\node (w4) at (-1.3,2.5) {\small{$7$}};
\node (w5) at (-1,1.2) {\small{$3$}};
\node (w6) at (1,1.2) {\small{$8$}};
\node (w7) at (2,1) {\small{$4$}};
\node (w8) at (3,2.5) {\small{$6$}};
\node (w9) at (5,2.5) {\small{$9$}};
\draw (v0) edge (w0);
\draw (v1) edge (v0);
\draw (v3) edge (v1);
\draw (v4) edge (v1);
\draw (v5) edge (v0);
\draw (v2) edge (v0);
\draw (v6) edge (v2);
\draw (w1) edge (v3);
\draw (w2) edge (v3);
\draw (w3) edge (v4);
\draw (w4) edge (v4);
\draw (w5) edge (v5);
\draw (w6) edge (v5);
\draw (w7) edge (v2);
\draw (w8) edge (v6);
\draw (w9) edge (v6);
\end{tikzpicture}
\quad \leftrightarrow \quad d(a(a(x_1,x_5),c(x_2,x_7)),c(x_3,x_8),c(x_4,b(x_6,x_9)))
\end{equation}
Any ordering of monomials compatible with compositions leads to the theory of Gr\"obner bases.
We refer for all details of \emph{Gr\"obner bases} and \emph{shuffle operads} to~\cite{DK::Grob}.

\begin{definition}[\cite{Loday_Ronco} Section~7.2]
A shuffle monomial $m$ in the free shuffle operad $\calF(\gamma_1,\ldots,\gamma_n)$ is called \underline{a left comb} 
%(or \emph{permutoidal})
 if it is presented by a shuffle tree with no nontrivial subtrees growing to the right. For example, if $\gamma_i$ are all binary operations then $m$ has to be of the following form:
\[\gamma_{i_1}(\gamma_{i_2}(\ldots \gamma_{i_k}(x_{1},x_{1+\sigma(1)}),\ldots x_{1+\sigma(k-1)}),x_{1+\sigma(k)}) =
\begin{tikzpicture}[scale=0.8]
\coordinate (v0) at (0,-2);
\node[ext] (v1) at (0,-1) {\tiny{$\gamma_{i_1}$}};
\node (w4) at (1,0) {\tiny{$1+\sigma(k)$}};
\node[ext] (v2) at (-1,0) {\tiny{$\gamma_{i_2}$}};
\node[ext] (v3) at (-3,2) {\tiny{$\gamma_{i_k}$}};
\coordinate (w3) at (-1.7,1.3);
\node (u) at (-2,1) {\ldots};
\coordinate (w5) at (-2.3,1.3);
\node (w2) at (0,1) {\tiny{$1+\sigma(k-1)$}};
\node (w1) at (-4,3) {\small{$1$}};
\node (w0) at (-2,3) {\tiny{$1+\sigma(1)$}};
\draw (v1) edge (v0);
\draw (v2) edge (v1);
\draw (w4) edge (v1);
\draw (u) edge (v2);
\draw (w2) edge (v2);
\draw (v3) edge (w5);
\draw (w1) edge (v3);
\draw (w0) edge (v3);
\end{tikzpicture} 	
\text{, where } \sigma\in\bbS_k.
\]	
\end{definition}

\begin{theorem}
\label{thm::U:PBW}
If an operad $\calP$ 
admits a Gr\"obner basis $G$ such that the set of leading monomials of $G$ are left comb shuffle monomials, then the universal enveloping functor $\UP$ satisfies the PBW property. 
If, in addition, the Gr\"obner basis $G$ is quadratic, then the conditions of Theorem~\ref{cor::Perm::UP::Koszul} are satisfied: the operad $\calP$ and the twisted associative algebra $\partial(\calP^{!}_{\oprd})$ are Koszul.
\end{theorem} 
In order to explain the proof of this theorem we will extend the theory of shuffle operads for $1\ttt2$-operads we defined in~\ref{def::12-oper}.

\subsubsection{Shuffle $1\ttt2$-operads and applications}
The notion of a shuffle operad and the theory of Gr\"obner bases can be generalized in a very straightforward way to the case of colored operads. In order to save the readers time and patience we will not present this theory here because the two-colored operads we consider are much easier to work out separately. 
Recall that we defined (Definition~\ref{def::12-oper}) $1\ttt2$-operads to be the operads on two colors such that the number of inputs and the outputs of the second color have to be the same.

\begin{definition}
\emph{A shuffle $1\ttt2$-operad} $\calP$ consists of two collections of vector spaces $\calP(n)^{1}$ and $\calP(n)^{2}$ together with compositions
\begin{gather*}
\circ_{\sigma}^{1}: \calP(k)^{1}\otimes \calP(n_1)^{1}\otimes\ldots\otimes\calP(n_k)^{1} \rightarrow \calP(n_1+\ldots+n_k)^{1} \\
\circ_{\sigma}^{2}: \calP(k)^{2}\otimes \calP(n_1)^{2}\otimes \calP(n_2)^{1} \ldots\otimes\calP(n_k)^{1} \rightarrow \calP(n_1+\ldots+n_k)^{2}  
\end{gather*}
indexed by surjections $\sigma: [1 n_1+\ldots +n_k] \twoheadrightarrow [1 k]$ that satisfy the shuffle condition \eqref{eq::shuffle}.
The composition rules satisfy operadic associativity.
\end{definition}
Pictorially, a shuffle $1\ttt2$-operad consists of operations of the first color that do not differ from the ordinary operad and of operations whose first input and the output are colored by the second color. Respectively, the second color in a shuffle monomial of the free shuffle $1\ttt2$-operad is a path from the leaf labelled by $1$ and the root. Here is an example of a $1\ttt2$-shuffle monomial:
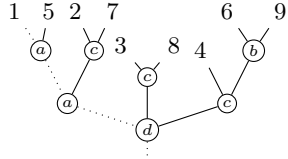
\[
\begin{tikzpicture}[scale=0.35]
\node[ext] (v0) at (0,-2) {\tiny{$d$}};
\node[ext] (v1) at (-3,-1) {\tiny{$a$}};
\node[ext] (v2) at (3,-1) {\tiny{$c$}};
\node[ext] (v3) at (-4,1) {\tiny{$a$}};
\node[ext] (v4) at (-2,1) {\tiny{$c$}};
\node[ext] (v5) at (0,0) {\tiny{$c$}};
\node[ext] (v6) at (4,1) {\tiny{$b$}};
\coordinate (w0) at (0,-3);
\node (w1) at (-5,2.5) {\small{$1$}};
\node (w2) at (-3.7,2.5) {\small{$5$}};
\node (w3) at (-2.7,2.5) {\small{$2$}};
\node (w4) at (-1.3,2.5) {\small{$7$}};
\node (w5) at (-1,1.2) {\small{$3$}};
\node (w6) at (1,1.2) {\small{$8$}};
\node (w7) at (2,1) {\small{$4$}};
\node (w8) at (3,2.5) {\small{$6$}};
\node (w9) at (5,2.5) {\small{$9$}};
\draw[dotted] (v0) edge (w0);
\draw[dotted] (v1) edge (v0);
\draw[dotted] (v3) edge (v1);
\draw (v4) edge (v1);
\draw (v5) edge (v0);
\draw (v2) edge (v0);
\draw (v6) edge (v2);
\draw[dotted] (w1) edge (v3);
\draw (w2) edge (v3);
\draw (w3) edge (v4);
\draw (w4) edge (v4);
\draw (w5) edge (v5);
\draw (w6) edge (v5);
\draw (w7) edge (v2);
\draw (w8) edge (v6);
\draw (w9) edge (v6);
\end{tikzpicture}
\]
In other words, the second color in a $1\ttt2$-shuffle monomial is uniquely defined and the only thing we have to know is the underlying shuffle monomial with coloring omitted.

\begin{proposition}
	There exists a linear ordering of $1\ttt2$-shuffle monomials in the free shuffle $1\ttt2$-operad that extends a given ordering of the generators and is compatible with the compositions.
\end{proposition}
\begin{proof}
Let $\calF$ be the free shuffle $1\ttt2$-operad generated by $S:=\{\alpha_1,\ldots,\alpha_s\}\subset \calF^{1}$ with the output of the first color and $T:=\{\beta_1,\ldots,\beta_t\}\in\calF^2$ with the output of the second color. Let $\overline{\calF}$ be the corresponding free shuffle operad with the set of generators $S\sqcup T$ where we just forget the colorings. 
Note that each $1\ttt2$-shuffle monomial in $\calF$ can be considered as an ordinary shuffle monomial of $\overline{\calF}$.
Thus we have an embedding of the monomial basis of the free shuffle $1\ttt2$-operad $\calF$ to the monomial basis of the free shuffle operad $\overline{\calF}$.  
Let us fix a linear ordering of the generators and let $\prec$ be a compatible ordering of the shuffle monomials in $\overline{\calF}$ which exist as shown in~\cite{DK::Grob}. The restriction of the linear ordering $\prec$ to the set of $1\ttt2$-shuffle monomials defines a compatible ordering.
\end{proof}
\begin{corollary}
	There exists a theory of Gr\"obner bases for shuffle $1\ttt2$-operads.
\end{corollary}

We have two natural (exact) functors with  shuffle $1\ttt2$-operads as a target category:

First, 
 with each symmetric $1\ttt2$-operad $\calQ$ we can associate a shuffle $1\ttt2$-operad $\shfl(\calQ)$ while forgetting the action of symmetric group and part of the compositions.
 
 Second, with each shuffle operad $\calP$ we can assign a shuffle $1\ttt2$-operad $\calP_{+}$ whose spaces of $n$-ary operations $\calP_{+}(n)^{1}=\calP_{+}(n,0)^{1}$, $\calP_{+}(n)^{2}=\calP(n-1,1)^{2}$ are both isomorphic to $\calP(n)$ and $1\ttt2$-compositions are restrictions of the underlying compositions in the shuffle operad $\calP$.
\begin{proposition}
The functors $\shfl$ and $+$ commute and, moreover, they map free objects to free objects:
	\[ 
	\begin{tikzcd}
	\text{Symmetric Operads} \arrow[r,"+"] \arrow[d,"\shfl"'] & 
	\text{Symmetric $1\ttt2$-Operads} \arrow[d,"\shfl"] \\
	\text{Shuffle Operads} \arrow[r,"+"] & 
	\text{Shuffle $1\ttt2$-Operads}
	\end{tikzcd}
	\]
	In particular, the functors $\shfl$ preserve the number of generators of a free symmetric operad  and the functor $+$ doubles the set of generators: $+(\calF(A))$ is generated by $(A,\partial A)$ -- one copy of the first color and another copy has the first input and the output of the second color.
\end{proposition}
\begin{proof}
	Direct observation. 
\end{proof}

\begin{corollary}
	Let $G$ be a Gr\"obner basis of the given shuffle operad $\calP = \calF(A|R)$. Then $G_{+}:=(G,\partial G)$ is a Gr\"obner basis of the shuffle $1\ttt2$-operad $\calP_{+}$ for an appropriate ordering of monomials in the free operad $\calF(A)_{+}=\calF(A_{+})$.
\end{corollary}
\begin{proof}
	We say that a $1\ttt2$-shuffle monomial $m$ is less than equal to $n$ if the underlying (uncolored) shuffle  monomials have the same comparability: $m\prec n$. This defines a linear ordering on monomials of the first color $\calF(A)_{+}^{1}$ and on monomials of the second color $\calF(A)_{+}^{2}$. Thus, in order to derive the theory of Gr\"obner basis we do not need the comparison of monomials with different coloring of the output.
	It is clear from the description of the comparability of colored monomials that the functor $+$ is compatible with filtrations given by the ordering of monomials.
\end{proof}

\subsubsection{Proof of Theorem~\ref{thm::U:PBW}}
\begin{proof}
Consider the filtration on the free shuffle operad and the induced filtration on $\calP=\calF(A|G)$ such that the associated graded operad $\gr\calP$ is an operad with monomial relations given by the leading terms of the Gr\"obner basis $G$.
Let $A_{+}^1$ and $A_{+}^{2}$ be the set of generators of $\calP_{+}$ with the output of the first and the second color correspondingly.
Let $\hat{G}_{+}^1$ (resp. $\hat{G}_{+}^2$) be the set of leading monomials of the Gr\"obner basis for $\calP_{+}$ with the output of the first (resp. second) color. Both sets $\hat{G}_{+}^1$ and $\hat{G}_{+}^2$ consists of left combs. 
Therefore, $\forall g_1\in\hat{G}_{+}^1$, $\forall g_2\in\hat{G}_{+}^2$ the greatest common divisor of $g_1$ and $g_2$ is $1$. (We refer to~\cite{DK::Grob}\S3.3 for the notion of divisibility.) 
In particular, there are no compositions between $\hat{G}_{+}^1$ and $\hat{G}_{+}^2$. 
Consequently, the associated graded $\gr\calP_{+}^{2}$ is isomorphic to
\[
\left(\calF_{\oprd}(A|\hat{G})\right)_{+}^{2} \simeq 
\left( \calF_{1\ttt2\oprd}(A_{+}^{1}\oplus A_{+}^{2}|\hat{G}_{+}^1\oplus\hat{G}_{+}^2) \right)^{2} \simeq 
\calF_{1\ttt2}(A_{+}^{2}|\hat{G}_{+}^2) \circ \calF_{\oprd}(A_{+}^{1}|\hat{G}_{+}^1) \simeq \calF_{1\ttt2}(A_{+}^{2}|\hat{G}_{+}^2)\circ \gr\calP
\]
Therefore, $\gr(\dcalP)$ is the free right $\gr\calP$-module.
The freeness for associated graded implies the freeness of $\dcalP$ for the initial operad $\calP$ and, hence, the PBW property of the universal enveloping functor $\UP$.

Note that the shuffle $1\ttt2$-suboperad $\partial\calF_{1\ttt2}(A_{+}^{2}|\hat{G}_{+}^2)$ of $\gr\calP$ is a shuffle algebra in the sence of~\cite{DK::Shuffle::Pattern}, that is the twisted associative algebra with forgotten action of symmetric groups. It follows that 
the shuffle algebra associated with the twisted associative algebra $\UP$ admits a filtration such that the associated graded is isomorphic to the shuffle algebra with monomial relations 
$\calF_{1\ttt2}(A_{+}^{2}|\hat{G}_{+}^2)$. Therefore, if the Gr\"obner basis $G$ is quadratic then the operad $\calP$ is Koszul by the results of~\cite{DK::Grob} and the twisted associative algebra $\UP$ is Koszul thanks to the results of~\cite{DK::Shuffle::Pattern}. 
\end{proof}

\section{Examples}
\label{sec::examples}

We consider a short list of examples of operads generated by binary operations and we refer to~\cite{Zinbiel} for more detailed descriptions of the corresponding operads.

\subsection{Pair of compatible Lie brackets}
\label{ex::Lie2}

Let $\Lie_2$ be the operad of pairs of compatible Lie brackets (introduced in~\cite{DK::Lie2}). An algebra $\frg$ over $\Lie_2$ has two Lie brackets $[\ttt,\ttt]_1,[\ttt,\ttt]_2:\Lambda^{2}\frg \rightarrow \frg$ such that each linear combination $\lambda[\ttt,\ttt]_1+\mu[\ttt,\ttt]_2$ defines a Lie bracket (satisfies the Jacobi identity).

Following the explicit construction of the universal enveloping algebra suggested in Corollary~\ref{prp::+Ex}, we have the following presentation of $\UU_{\Lie_2}(\mathfrak{g})$:
\[
\UU_{\Lie_2}(\mathfrak{g}) \simeq
\kk\left\langle \mathfrak{g}_1\oplus \mathfrak{g}_2 \left|
\begin{array}{c}
g_1\otimes h_1 - h_1 \otimes g_1 =[g,h]_{1}, \\
{
\begin{array}{rl}
g_1\otimes h_2 + g_2 \otimes h_1 & - h_1\otimes g_1 - h_2\otimes g_2 = \\
& = [g,h]_{1} + [g,h]_2, 
\end{array}
}
\\
{ g_2\otimes h_2 - h_2 \otimes g_2 = [g,h]_{2} }
\\
\text{ here } g_i, h_i \in \mathfrak{g}_i 
\end{array}
\right.\right\rangle
\]
\begin{lemma}
The universal enveloping functor $\UU_{\Lie_2}$ satisfies PBW and 
 the associated graded algebra 
$\gr^{PBW}\UU_{\Lie_2}(\mathfrak{g})$ is isomorphic to the quotient of the free associative algebra  $T^{\otimes}(\mathfrak{g}\otimes \kk^2)$ by the ideal generated by 
$$\Lambda^2(\mathfrak{g})\otimes S^2(\kk^2) \subset (\mathfrak{g})^{\otimes 2} (\kk^2)^{\otimes 2} \simeq (\mathfrak{g}\otimes \kk^2)^{\otimes 2} \simeq T^{\otimes 2}(\mathfrak{g}\otimes \kk^2)$$
 Moreover,
\[
\rchi_{\UU_{\Lie_2}^0} = \frac{\exp\left(\sum_{k\geq 1} \frac{p_k}{k}\right)}{1-(\sum_{k\geq 1} p_k)} = \frac{\sum_{k\geq 1} h_k}{{1-(\sum_{k\geq 1} p_k)}}.
\]		
\end{lemma}
\begin{proof}
The operad Koszul dual to $\Lie_2$ is a particular example of the construction considered in Example~\ref{ex::Alg2Op}. Indeed, we have an isomorphism of operads $\calO_{\kk[x,y]} \simeq \Lie_2^!$.
Thus, thanks to Corollary~\ref{cor::Perm::UP::Koszul} the functor $\UU_{\Lie_2}$ satisfies PBW.

For the computation of the character $\rchi_{\UU_{\Lie_2}}$ it is enough to recall that the space of $n$-ary operations in the Koszul dual operad to the operad $\Lie_2$ is a trivial $\bbS_n$-representation of dimension $n$.
		Thus its generating series is very simple (\cite{DK::Lie2}):
\[\rchi_{\Lie_2^{!}} = \left(\sum_{k\geq 1} p_k\right) \exp\left(\sum_{k\geq 1} \frac{p_k}{k}\right). 
\]
The substitution of $\rchi_{\Lie_2^{!}}$ into Equation~\eqref{eq::UPZ::character} leads to the following:
\begin{multline*}
\rchi_{\UU_{\Lie_2}^0} =
- \left[\frac{\partial}{\partial p_1} \rchi_{\calO_{\kk[x,y]}}(-p_1,\ldots)\right]^{-1} = 
\left[\exp\left(-\sum_{k\geq 1} \frac{p_k}{k}\right)\left(1-\sum_{k\geq 1} \frac{p_k}{k}\right)\right]^{-1} = \frac{\exp\left(\sum_{k\geq 1} \frac{p_k}{k}\right)}{1-\sum_{k\geq 1} \frac{p_k}{k}}.
\end{multline*}
\end{proof}
Computer experiments show that the expression $({1-(\sum_{k\geq 1} p_k)})^{-1}$ is Schur positive and we conjecture that it can be expressed in terms of the free algebra over an operad. Denote the corresponding Schur functor by $FL()$ we expect an isomorphism of Schur functors:
\[
U_{\Lie_2}^{0}(V) \simeq S(V) \otimes FL(V). 
\]

\subsection{PreLie algebras}
\label{ex::PreLie}

The $\PreLie$ operad of Pre-Lie algebras introduced in~\cite{Chapoton} is a well-known example of a Koszul operad. It is generated by one nonsymmetric operation $x\triangleright y$ yielding the following condition:
\begin{equation}
\label{eq::Prelie}
(x\triangleright y)\triangleright z - x\triangleright(y\triangleright z) = (x\triangleright z)\triangleright y - x\triangleright(z\triangleright y)
\end{equation}
It might be easier to define a universal enveloping algebra directly from relation~\eqref{eq::Prelie}.
The universal enveloping algebra $\UU_{\PreLie}(\frg)$ of a PreLie algebra $\frg$ is generated by two copies of $\frg$ (whose elements are denoted by $r_{x}$, $l_x$ and represent multiplication from the right and from the left respectively). These elements should satisfy the following list of relations for all pairs of $x,y\in\frg$:
\begin{gather*}
r_x r_y - r_y r_x = r_{x\triangleright y} - r_{y\triangleright x}, \\
l_x r_y - r_y l_x = l_y l_x - l_{x\triangleright y}.
\end{gather*}

\begin{theorem}
\label{thm::PreLie}	
The functor $\UU_{\PreLie}$ satisfies PBW. 
Moreover, on top of the PBW-filtration there exists a filtration on $\UU_{\PreLie}(\frg)$ such that 
the associated graded algebra is isomorphic to $S(\frg)\otimes T(\frg)$. 
\end{theorem}
\begin{proof}
This theorem can be derived from the description of the universal enveloping in terms of generators and relations. However, we want to illustrate our methods and deal with relations in the operad.
Note, that the shuffle operad $\PreLie$ is generated by two operations $x_1\triangleright x_2$ and $x_1\triangleleft x_2 = x_2\triangleright x_1$ and we know (Example 11 in~\cite{DK::Grob}) that there exists a compatible ordering of monomials such that
the quadratic Gr\"obner basis for $\PreLie$ consists of $3$ relations (where we underline the leading monomial in each relation)
\begin{gather*}
\underline{
\begin{tikzpicture}[scale=0.38]
\node[ext] (v1) at (0,-1) {\tiny{$\triangleright$}};
\node[ext] (v2) at (-1,1) {\tiny{$\triangleright$}};
\coordinate (r) at (0,-2);
\node (w1) at (-1.5,2.5) {${\tiny{1}}$};
\node (w2) at (-0.5,2.5) {${\tiny{2}}$};
\node (w3) at (0.5,1) {${\tiny{3}}$};
\draw (v1) edge (r);
\draw (v1) edge (v2) edge (w3);
\draw (v2) edge (w1) edge (w2);
\end{tikzpicture}} -
\begin{tikzpicture}[scale=0.38]
\node[ext] (v1) at (0,-1) {\tiny{$\triangleright$}};
\node[ext] (v2) at (1,1) {\tiny{$\triangleright$}};
\coordinate (r) at (0,-2);
\node (w1) at (1.5,2.5) {${\tiny{3}}$};
\node (w2) at (0.5,2.5) {${\tiny{2}}$};
\node (w3) at (-0.5,1) {${\tiny{1}}$};
\draw (v1) edge (r);
\draw (v1) edge (v2) edge (w3);
\draw (v2) edge (w1) edge (w2);
\end{tikzpicture}=
\begin{tikzpicture}[scale=0.38]
\node[ext] (v1) at (0,-1) {\tiny{$\triangleright$}};
\node[ext] (v2) at (-1,1) {\tiny{$\triangleright$}};
\coordinate (r) at (0,-2);
\node (w1) at (-1.5,2.5) {${\tiny{1}}$};
\node (w2) at (-0.5,2.5) {${\tiny{3}}$};
\node (w3) at (0.5,1) {${\tiny{2}}$};
\draw (v1) edge (r);
\draw (v1) edge (v2) edge (w3);
\draw (v2) edge (w1) edge (w2);
\end{tikzpicture} -
\begin{tikzpicture}[scale=0.38]
\node[ext] (v1) at (0,-1) {\tiny{$\triangleright$}};
\node[ext] (v2) at (1,1) {\tiny{$\triangleleft$}};
\coordinate (r) at (0,-2);
\node (w1) at (1.5,2.5) {${\tiny{3}}$};
\node (w2) at (0.5,2.5) {${\tiny{2}}$};
\node (w3) at (-0.5,1) {${\tiny{1}}$};
\draw (v1) edge (r);
\draw (v1) edge (v2) edge (w3);
\draw (v2) edge (w1) edge (w2);
\end{tikzpicture}; \\
\underline{
	\begin{tikzpicture}[scale=0.38]
	\node[ext] (v1) at (0,-1) {\tiny{$\triangleright$}};
	\node[ext] (v2) at (-1,1) {\tiny{$\triangleleft$}};
	\coordinate (r) at (0,-2);
	\node (w1) at (-1.5,2.5) {${\tiny{1}}$};
	\node (w2) at (-0.5,2.5) {${\tiny{2}}$};
	\node (w3) at (0.5,1) {${\tiny{3}}$};
	\draw (v1) edge (r);
	\draw (v1) edge (v2) edge (w3);
	\draw (v2) edge (w1) edge (w2);
	\end{tikzpicture}} -
\begin{tikzpicture}[scale=0.38]
\node[ext] (v1) at (0,-1) {\tiny{$\triangleleft$}};
\node[ext] (v2) at (-1,1) {\tiny{$\triangleright$}};
\coordinate (r) at (0,-2);
\node (w1) at (-1.5,2.5) {${\tiny{1}}$};
\node (w2) at (-0.5,2.5) {${\tiny{3}}$};
\node (w3) at (0.5,1) {${\tiny{2}}$};
\draw (v1) edge (r);
\draw (v1) edge (v2) edge (w3);
\draw (v2) edge (w1) edge (w2);
\end{tikzpicture}
=
\begin{tikzpicture}[scale=0.38]
\node[ext] (v1) at (0,-1) {\tiny{$\triangleleft$}};
\node[ext] (v2) at (1,1) {\tiny{$\triangleright$}};
\coordinate (r) at (0,-2);
\node (w1) at (1.5,2.5) {${\tiny{3}}$};
\node (w2) at (0.5,2.5) {${\tiny{2}}$};
\node (w3) at (-0.5,1) {${\tiny{1}}$};
\draw (v1) edge (r);
\draw (v1) edge (v2) edge (w3);
\draw (v2) edge (w1) edge (w2);
\end{tikzpicture}
-
\begin{tikzpicture}[scale=0.38]
\node[ext] (v1) at (0,-1) {\tiny{$\triangleleft$}};
\node[ext] (v2) at (-1,1) {\tiny{$\triangleleft$}};
\coordinate (r) at (0,-2);
\node (w1) at (-1.5,2.5) {${\tiny{1}}$};
\node (w2) at (-0.5,2.5) {${\tiny{3}}$};
\node (w3) at (0.5,1) {${\tiny{2}}$};
\draw (v1) edge (r);
\draw (v1) edge (v2) edge (w3);
\draw (v2) edge (w1) edge (w2);
\end{tikzpicture};
\quad
{
	\begin{tikzpicture}[scale=0.38]
	\node[ext] (v1) at (0,-1) {\tiny{$\triangleleft$}};
	\node[ext] (v2) at (-1,1) {\tiny{$\triangleright$}};
	\coordinate (r) at (0,-2);
	\node (w1) at (-1.5,2.5) {${\tiny{1}}$};
	\node (w2) at (-0.5,2.5) {${\tiny{2}}$};
	\node (w3) at (0.5,1) {${\tiny{3}}$};
	\draw (v1) edge (r);
	\draw (v1) edge (v2) edge (w3);
	\draw (v2) edge (w1) edge (w2);
	\end{tikzpicture}} -
\underline{
\begin{tikzpicture}[scale=0.38]
\node[ext] (v1) at (0,-1) {\tiny{$\triangleright$}};
\node[ext] (v2) at (-1,1) {\tiny{$\triangleleft$}};
\coordinate (r) at (0,-2);
\node (w1) at (-1.5,2.5) {${\tiny{1}}$};
\node (w2) at (-0.5,2.5) {${\tiny{3}}$};
\node (w3) at (0.5,1) {${\tiny{2}}$};
\draw (v1) edge (r);
\draw (v1) edge (v2) edge (w3);
\draw (v2) edge (w1) edge (w2);
\end{tikzpicture}
}
=
\begin{tikzpicture}[scale=0.38]
\node[ext] (v1) at (0,-1) {\tiny{$\triangleleft$}};
\node[ext] (v2) at (-1,1) {\tiny{$\triangleleft$}};
\coordinate (r) at (0,-2);
\node (w1) at (-1.5,2.5) {${\tiny{1}}$};
\node (w2) at (-0.5,2.5) {${\tiny{2}}$};
\node (w3) at (0.5,1) {${\tiny{3}}$};
\draw (v1) edge (r);
\draw (v1) edge (v2) edge (w3);
\draw (v2) edge (w1) edge (w2);
\end{tikzpicture}
-
\begin{tikzpicture}[scale=0.38]
\node[ext] (v1) at (0,-1) {\tiny{$\triangleleft$}};
\node[ext] (v2) at (1,1) {\tiny{$\triangleleft$}};
\coordinate (r) at (0,-2);
\node (w1) at (1.5,2.5) {${\tiny{3}}$};
\node (w2) at (0.5,2.5) {${\tiny{2}}$};
\node (w3) at (-0.5,1) {${\tiny{1}}$};
\draw (v1) edge (r);
\draw (v1) edge (v2) edge (w3);
\draw (v2) edge (w1) edge (w2);
\end{tikzpicture}.
\end{gather*}
Moreover, we claim that one can define a filtration on the operad $\PreLie$ such that the associated graded is still a quadratic operad subject to the following relations that form a Gr\"obner basis:
\[
\underline{
	\begin{tikzpicture}[scale=0.38]
	\node[ext] (v1) at (0,-1) {\tiny{$\triangleright$}};
	\node[ext] (v2) at (-1,1) {\tiny{$\triangleright$}};
	\coordinate (r) at (0,-2);
	\node (w1) at (-1.5,2.5) {${\tiny{1}}$};
	\node (w2) at (-0.5,2.5) {${\tiny{2}}$};
	\node (w3) at (0.5,1) {${\tiny{3}}$};
	\draw (v1) edge (r);
	\draw (v1) edge (v2) edge (w3);
	\draw (v2) edge (w1) edge (w2);
	\end{tikzpicture}} -
\begin{tikzpicture}[scale=0.38]
\node[ext] (v1) at (0,-1) {\tiny{$\triangleright$}};
\node[ext] (v2) at (-1,1) {\tiny{$\triangleright$}};
\coordinate (r) at (0,-2);
\node (w1) at (-1.5,2.5) {${\tiny{1}}$};
\node (w2) at (-0.5,2.5) {${\tiny{3}}$};
\node (w3) at (0.5,1) {${\tiny{2}}$};
\draw (v1) edge (r);
\draw (v1) edge (v2) edge (w3);
\draw (v2) edge (w1) edge (w2);
\end{tikzpicture} =
\underline{
	\begin{tikzpicture}[scale=0.38]
	\node[ext] (v1) at (0,-1) {\tiny{$\triangleright$}};
	\node[ext] (v2) at (-1,1) {\tiny{$\triangleleft$}};
	\coordinate (r) at (0,-2);
	\node (w1) at (-1.5,2.5) {${\tiny{1}}$};
	\node (w2) at (-0.5,2.5) {${\tiny{2}}$};
	\node (w3) at (0.5,1) {${\tiny{3}}$};
	\draw (v1) edge (r);
	\draw (v1) edge (v2) edge (w3);
	\draw (v2) edge (w1) edge (w2);
	\end{tikzpicture}} -
\begin{tikzpicture}[scale=0.38]
\node[ext] (v1) at (0,-1) {\tiny{$\triangleleft$}};
\node[ext] (v2) at (-1,1) {\tiny{$\triangleright$}};
\coordinate (r) at (0,-2);
\node (w1) at (-1.5,2.5) {${\tiny{1}}$};
\node (w2) at (-0.5,2.5) {${\tiny{3}}$};
\node (w3) at (0.5,1) {${\tiny{2}}$};
\draw (v1) edge (r);
\draw (v1) edge (v2) edge (w3);
\draw (v2) edge (w1) edge (w2);
\end{tikzpicture}
=
{
	\begin{tikzpicture}[scale=0.38]
	\node[ext] (v1) at (0,-1) {\tiny{$\triangleleft$}};
	\node[ext] (v2) at (-1,1) {\tiny{$\triangleright$}};
	\coordinate (r) at (0,-2);
	\node (w1) at (-1.5,2.5) {${\tiny{1}}$};
	\node (w2) at (-0.5,2.5) {${\tiny{2}}$};
	\node (w3) at (0.5,1) {${\tiny{3}}$};
	\draw (v1) edge (r);
	\draw (v1) edge (v2) edge (w3);
	\draw (v2) edge (w1) edge (w2);
	\end{tikzpicture}} -
\underline{
	\begin{tikzpicture}[scale=0.38]
	\node[ext] (v1) at (0,-1) {\tiny{$\triangleright$}};
	\node[ext] (v2) at (-1,1) {\tiny{$\triangleleft$}};
	\coordinate (r) at (0,-2);
	\node (w1) at (-1.5,2.5) {${\tiny{1}}$};
	\node (w2) at (-0.5,2.5) {${\tiny{3}}$};
	\node (w3) at (0.5,1) {${\tiny{2}}$};
	\draw (v1) edge (r);
	\draw (v1) edge (v2) edge (w3);
	\draw (v2) edge (w1) edge (w2);
	\end{tikzpicture}
}= 0
\]
The corresponding filtration descends to the filtration on $\UU_{\PreLie}(\frg)$ and its associated graded is isomorphic to $S(\frg)\otimes T(\frg)$.
\end{proof}

As predicted by Theorem~\ref{cor::Perm::UP::Koszul} the PBW property of $\UU_\PreLie$ implies the simplification of the cohomology theory of the category of modules over a given $\PreLie$-algebra.
Certain cohomological complexes were presented in~\cite{Dzhumadildaev} in 1999. We claim that our proof and description are enough to recover all corresponding cohomology theories, however we will not go into details here.

While this preprint was under review an alternative proof of Theorem~\ref{thm::PreLie} was posted on the arXiv in~\cite{Dots::UPreLie}.

\subsection{Operad $\Perm=\PreLie^{!}$}
\label{ex::Perm}
The Koszul dual operad to the operad $\PreLie$ of pre-Lie algebras is known under the the name  $\Perm$ (\cite{Chapoton_Perm},\cite{Zinbiel}). 
The operad $\Perm$ is generated by two operations for which we will use the same notations as for $\PreLie$:
\[ 
x\triangleright y = y \triangleleft x 
\]
subject to the following quadratic relations:
\[
(x \triangleright y)\triangleright z = x \triangleright (y\triangleright z ) = x \triangleright (z \triangleright y) 
%= x \triagleright (y\triangleleft z )
\]
We claim that the universal enveloping functor $\UU_{\Perm}$ does not satisfy PBW thanks to the following lemma and the main criterion stated in Theorem~\ref{thm::U:PBW}.
\begin{lemma}
	The twisted associative algebra $\partial\PreLie$ and the operad $\PreLie$ have different sets of generators.
\end{lemma}
\begin{proof}
Note that the space of $n$-ary operations in the free twisted associative algebra $\calF(V)$ generated by the vector space ($\bbS_1$-representation) $V\subset \calF(1)$ is isomorphic to $\kk[\bbS_n]\otimes V^{\otimes n}$. In particular, the dimension of the space of $2$-ary operations of the free twisted associative algebra on $2$ binary generators is $2!\cdot 2^2 =8$.

On the other hand, there are two binary generators of the operad $\PreLie$ and therefore, there exist two generators in the twisted associative algebra $\partial\PreLie$ that belong to the subspace $\partial\PreLie(1) = \partial\PreLie(1)$. However, 
\[\dim(\partial\PreLie(2)) = \dim(\PreLie(3)) = 9 > 8.\] 
\end{proof}

\subsection{Poisson algebras}
\label{sec::Pois}

We have seen in Example~\ref{ex::Poisson} that the universal enveloping functor of the Poisson operad $\Pois$ does not satisfy PBW. 
However, we insist that the homological description of this functor is of particular interest and we will discuss it in detail elsewhere.
The description of the twisted associative algebra  $\partial\Pois$ is the first step in this direction that illustrates the complexity of the theory of twisted associative algebras.

\begin{proposition}
The twisted associative algebra $\partial\Pois$ is generated by two elements
$\begin{tikzpicture}[scale=0.5]
\node[int] (v1) at (0,0) {\phantom{.}};
\coordinate (v0) at (-1,0);
\coordinate (v2) at (1,0);
\node (w) at (0,1) {\tiny{$1$}};
\draw[dotted] (v1) edge (v0) edge (v2);
\draw (v1) edge (w);	
\end{tikzpicture}
$ (representing the image of the commutative product)	and 
$\begin{tikzpicture}[scale=0.5]
\node[ext] (v1) at (0,0) {\phantom{.}};
\coordinate (v0) at (-1,0);
\coordinate (v2) at (1,0);
\node (w) at (0,1) {\tiny{$1$}};
\draw[dotted] (v1) edge (v0) edge (v2);
\draw (v1) edge (w);	
\end{tikzpicture}
$ (representing the image of the Lie bracket)
yielding the following quadratic relations:
\begin{gather}
\label{eq::Pois::perm1}
\begin{tikzpicture}[scale=0.5]
\node[int] (v1) at (0,0) {\phantom{.}};
\coordinate (v0) at (-1,0);
\node[int] (v2) at (1,0) {\phantom{.}};
\node (w1) at (0,1) {\tiny{$1$}};
\coordinate (v3) at (2,0);
\node (w2) at (1,1) {\tiny{$2$}};
\draw[dotted] (v1) edge (v0) edge (v2);
\draw[dotted] (v2) edge (v3);
\draw (v1) edge (w1);	
\draw (v2) edge (w2);	
\end{tikzpicture}
= 
\begin{tikzpicture}[scale=0.5]
\node[int] (v1) at (0,0) {\phantom{.}};
\coordinate (v0) at (-1,0);
\node[int] (v2) at (1,0) {\phantom{.}};
\node (w1) at (0,1) {\tiny{$2$}};
\coordinate (v3) at (2,0);
\node (w2) at (1,1) {\tiny{$1$}};
\draw[dotted] (v1) edge (v0) edge (v2);
\draw[dotted] (v2) edge (v3);
\draw (v1) edge (w1);	
\draw (v2) edge (w2);	
\end{tikzpicture} \\
\begin{tikzpicture}[scale=0.5]
\node[ext] (v1) at (0,0) {\phantom{.}};
\coordinate (v0) at (-1,0);
\node[int] (v2) at (1,0) {\phantom{.}};
\node (w1) at (0,1) {\tiny{$1$}};
\coordinate (v3) at (2,0);
\node (w2) at (1,1) {\tiny{$2$}};
\draw[dotted] (v1) edge (v0) edge (v2);
\draw[dotted] (v2) edge (v3);
\draw (v1) edge (w1);	
\draw (v2) edge (w2);	
\end{tikzpicture}
-
\begin{tikzpicture}[scale=0.5]
\node[int] (v1) at (0,0) {\phantom{.}};
\coordinate (v0) at (-1,0);
\node[ext] (v2) at (1,0) {\phantom{.}};
\node (w1) at (0,1) {\tiny{$2$}};
\coordinate (v3) at (2,0);
\node (w2) at (1,1) {\tiny{$1$}};
\draw[dotted] (v1) edge (v0) edge (v2);
\draw[dotted] (v2) edge (v3);
\draw (v1) edge (w1);	
\draw (v2) edge (w2);	
\end{tikzpicture} = -
\begin{tikzpicture}[scale=0.5]
\node[ext] (v1) at (0,0) {\phantom{.}};
\coordinate (v0) at (-1,0);
\node[int] (v2) at (1,0) {\phantom{.}};
\node (w1) at (0,1) {\tiny{$2$}};
\coordinate (v3) at (2,0);
\node (w2) at (1,1) {\tiny{$1$}};
\draw[dotted] (v1) edge (v0) edge (v2);
\draw[dotted] (v2) edge (v3);
\draw (v1) edge (w1);	
\draw (v2) edge (w2);	
\end{tikzpicture}
+
\begin{tikzpicture}[scale=0.5]
\node[int] (v1) at (0,0) {\phantom{.}};
\coordinate (v0) at (-1,0);
\node[ext] (v2) at (1,0) {\phantom{.}};
\node (w1) at (0,1) {\tiny{$1$}};
\coordinate (v3) at (2,0);
\node (w2) at (1,1) {\tiny{$2$}};
\draw[dotted] (v1) edge (v0) edge (v2);
\draw[dotted] (v2) edge (v3);
\draw (v1) edge (w1);	
\draw (v2) edge (w2);	
\end{tikzpicture} 
\label{eq::Pois::perm2}
\end{gather}
and many other relations in higher arities.
\end{proposition}
\begin{proof}
Note that the colored operad $\partial\Ass$ admits a filtration by the number of commutators such that the associated graded is isomorphic to $\Pois$. Therefore, the same filtration exists on the twisted associative algebra $\partial\Ass$.
We checked already, that $\UU_{\Ass}$ satisfies PBW (Example~\ref{ex::Ass}). Consequently, thanks to Theorem~\ref{cor::Perm::UP::Koszul} the twisted associative algebra $\partial\Ass$ is generated by  $\partial\Ass(1)$ subject to quadratic relations and is Koszul. We can rewrite these relations with respect to the basis:
\[
\begin{tikzpicture}[scale=0.5]
\node[ext] (v1) at (0,0) {\phantom{.}};
\coordinate (v0) at (-1,0);
\coordinate (v2) at (1,0);
\node (w) at (0,1) {\tiny{$1$}};
\draw[dotted] (v1) edge (v0) edge (v2);
\draw (v1) edge (w);	
\end{tikzpicture} :=
\begin{tikzpicture}[scale=0.5]
\node[ext] (v1) at (0,0) {\tiny{$+$}};
\coordinate (v0) at (-1,0);
\coordinate (v2) at (1,0);
\node (w) at (0,1) {\tiny{$1$}};
\draw[dotted] (v1) edge (v0) edge (v2);
\draw (v1) edge (w);	
\end{tikzpicture} 
-
\begin{tikzpicture}[scale=0.5]
\node[ext] (v1) at (0,0) {\tiny{$-$}};
\coordinate (v0) at (-1,0);
\coordinate (v2) at (1,0);
\node (w) at (0,1) {\tiny{$1$}};
\draw[dotted] (v1) edge (v0) edge (v2);
\draw (v1) edge (w);	
\end{tikzpicture}, 
\quad 
\begin{tikzpicture}[scale=0.5]
\node[int] (v1) at (0,0) {\phantom{.}};
\coordinate (v0) at (-1,0);
\coordinate (v2) at (1,0);
\node (w) at (0,1) {\tiny{$1$}};
\draw[dotted] (v1) edge (v0) edge (v2);
\draw (v1) edge (w);	
\end{tikzpicture} :=
\begin{tikzpicture}[scale=0.5]
\node[ext] (v1) at (0,0) {\tiny{$+$}};
\coordinate (v0) at (-1,0);
\coordinate (v2) at (1,0);
\node (w) at (0,1) {\tiny{$1$}};
\draw[dotted] (v1) edge (v0) edge (v2);
\draw (v1) edge (w);	
\end{tikzpicture} 
+
\begin{tikzpicture}[scale=0.5]
\node[ext] (v1) at (0,0) {\tiny{$-$}};
\coordinate (v0) at (-1,0);
\coordinate (v2) at (1,0);
\node (w) at (0,1) {\tiny{$1$}};
\draw[dotted] (v1) edge (v0) edge (v2);
\draw (v1) edge (w);	
\end{tikzpicture}.
\]
and get the following:
\begin{gather*}
\begin{tikzpicture}[scale=0.5]
\node[int] (v1) at (0,0) {\phantom{.}};
\coordinate (v0) at (-1,0);
\node[int] (v2) at (1,0) {\phantom{.}};
\node (w1) at (0,1) {\tiny{$1$}};
\coordinate (v3) at (2,0);
\node (w2) at (1,1) {\tiny{$2$}};
\draw[dotted] (v1) edge (v0) edge (v2);
\draw[dotted] (v2) edge (v3);
\draw (v1) edge (w1);	
\draw (v2) edge (w2);	
\end{tikzpicture}
- 
\begin{tikzpicture}[scale=0.5]
\node[int] (v1) at (0,0) {\phantom{.}};
\coordinate (v0) at (-1,0);
\node[int] (v2) at (1,0) {\phantom{.}};
\node (w1) at (0,1) {\tiny{$2$}};
\coordinate (v3) at (2,0);
\node (w2) at (1,1) {\tiny{$1$}};
\draw[dotted] (v1) edge (v0) edge (v2);
\draw[dotted] (v2) edge (v3);
\draw (v1) edge (w1);	
\draw (v2) edge (w2);	
\end{tikzpicture} =
\frac{1}{2}\left(
\begin{tikzpicture}[scale=0.5]
\node[int] (v1) at (0,0) {\phantom{.}};
\coordinate (v0) at (-1,0);
\node[ext] (v2) at (1,0) {\phantom{.}};
\node (w1) at (0,1) {\tiny{$1$}};
\coordinate (v3) at (2,0);
\node (w2) at (1,1) {\tiny{$2$}};
\draw[dotted] (v1) edge (v0) edge (v2);
\draw[dotted] (v2) edge (v3);
\draw (v1) edge (w1);	
\draw (v2) edge (w2);	
\end{tikzpicture} + 
\begin{tikzpicture}[scale=0.5]
\node[ext] (v1) at (0,0) {\phantom{.}};
\coordinate (v0) at (-1,0);
\node[ext] (v2) at (1,0) {\phantom{.}};
\node (w1) at (0,1) {\tiny{$1$}};
\coordinate (v3) at (2,0);
\node (w2) at (1,1) {\tiny{$2$}};
\draw[dotted] (v1) edge (v0) edge (v2);
\draw[dotted] (v2) edge (v3);
\draw (v1) edge (w1);	
\draw (v2) edge (w2);	
\end{tikzpicture}
\right),
\\
\begin{tikzpicture}[scale=0.5]
\node[ext] (v1) at (0,0) {\phantom{.}};
\coordinate (v0) at (-1,0);
\node[int] (v2) at (1,0) {\phantom{.}};
\node (w1) at (0,1) {\tiny{$1$}};
\coordinate (v3) at (2,0);
\node (w2) at (1,1) {\tiny{$2$}};
\draw[dotted] (v1) edge (v0) edge (v2);
\draw[dotted] (v2) edge (v3);
\draw (v1) edge (w1);	
\draw (v2) edge (w2);	
\end{tikzpicture}
-
\begin{tikzpicture}[scale=0.5]
\node[int] (v1) at (0,0) {\phantom{.}};
\coordinate (v0) at (-1,0);
\node[ext] (v2) at (1,0) {\phantom{.}};
\node (w1) at (0,1) {\tiny{$2$}};
\coordinate (v3) at (2,0);
\node (w2) at (1,1) {\tiny{$1$}};
\draw[dotted] (v1) edge (v0) edge (v2);
\draw[dotted] (v2) edge (v3);
\draw (v1) edge (w1);	
\draw (v2) edge (w2);	
\end{tikzpicture} = -
\begin{tikzpicture}[scale=0.5]
\node[ext] (v1) at (0,0) {\phantom{.}};
\coordinate (v0) at (-1,0);
\node[int] (v2) at (1,0) {\phantom{.}};
\node (w1) at (0,1) {\tiny{$2$}};
\coordinate (v3) at (2,0);
\node (w2) at (1,1) {\tiny{$1$}};
\draw[dotted] (v1) edge (v0) edge (v2);
\draw[dotted] (v2) edge (v3);
\draw (v1) edge (w1);	
\draw (v2) edge (w2);	
\end{tikzpicture}
+
\begin{tikzpicture}[scale=0.5]
\node[int] (v1) at (0,0) {\phantom{.}};
\coordinate (v0) at (-1,0);
\node[ext] (v2) at (1,0) {\phantom{.}};
\node (w1) at (0,1) {\tiny{$1$}};
\coordinate (v3) at (2,0);
\node (w2) at (1,1) {\tiny{$2$}};
\draw[dotted] (v1) edge (v0) edge (v2);
\draw[dotted] (v2) edge (v3);
\draw (v1) edge (w1);	
\draw (v2) edge (w2);	
\end{tikzpicture} 
\label{eq::Pois::perm2}
\end{gather*}
We see that the relations~\eqref{eq::Pois::perm1},\eqref{eq::Pois::perm2} are the corresponding associated graded relations. Computer experiments suggested by V.Dotsenko show that there is at least an additional relation in arity $4$ that does not follow from the given quadratic relations.

The twisted associative algebra $\partial\Pois$ is not Koszul, because the Koszul operad $\Pois$ is selfdual and the functor $\UU_{(\Pois)^{!}_{\oprd}} = \UU_{\Pois}$ does not satisfy PBW as we showed in Example~\ref{ex::Poisson}.
\end{proof}

\subsection{Leibniz algebras}
\label{ex::Leib}

The operad of Leibniz algebras was introduced by J.-L.\,Loday (\cite{Loday_Cyclic}) as an interesting generalization of the operad of Lie algebras motivated by certain algebraic structures one can find on the cyclic homology of any associative algebra.
A vector space $V$, equipped with a binary (nonsymmetric) operation $[x,y]$ (called a bracket) yielding the identity:
\begin{equation}
\label{eq::Leibn::Oper::id}
\forall x,y,z\in V \quad [x,[y,z]]- [[x,y],z] + [[x,z],y]=0 
\end{equation}
is called an Leibniz algebra. 
One can easily see that the following identity holds in any Leibniz algebra:
\[
\forall x,y,z\in V \quad [x,[y,z]]+[x[z,y]]=0.
\]
\begin{proposition}
\label{prp::Leib::Grob}	
	The corresponding shuffle operad $\Leib$ is generated by two binary operations $x_1\triangleleft x_2 := [x_1,x_2]$ and 
	$x_1\triangleright x_2 := [x_2,x_1]$ 
	and the following list of quadratic relations:
\begin{gather*}
\underline{\rrtree{\triangleleft}{\triangleleft}} - \lltree{\triangleleft}{\triangleleft} + \lrtree{\triangleleft}{\triangleleft} =
\underline{\rrtree{\triangleleft}{\triangleright}} - \lltree{\triangleleft}{\triangleleft} + \lrtree{\triangleleft}{\triangleleft} =0;
\\
\underline{\rrtree{\triangleright}{\triangleleft}} - \lltree{\triangleleft}{\triangleright} + \lrtree{\triangleright}{\triangleright} =
\underline{\rrtree{\triangleright}{\triangleright}}  +\lltree{\triangleright}{\triangleright} - \lrtree{\triangleleft}{\triangleright} =0; \\
\underline{\lltree{\triangleright}{\triangleleft}} + \lltree{\triangleright}{\triangleright} = 
\underline{\lrtree{\triangleright}{\triangleleft}} + \lrtree{\triangleright}{\triangleright} = 0
\end{gather*}	
form a Gr\"obner basis with respect to convention $\binar{\triangleright} >\binar{\triangleleft}$
and
the path opposite-degree lexicographical ordering of shuffle monomials.
I.e. in order to compare two shuffle monomials $v$ and $w$ we first consider the paths (words in generators $\triangleleft$ and $\triangleright$) that goes from the root to the leaf $1$ and say that a monomial $v>w$, if the corresponding word assigned to $v$ is  \emph{less} then the corresponding word assigned to $w$ in the degree-lexicographical ordering of words in two letters.

In particular, the operad $\Leib$ is Koszul.
\end{proposition} 

In~\cite{Loday_Pir} J.-L.\,Loday and T.\,Pirashvili discussed the structure of a universal enveloping algebra of a Leibniz algebra.
\begin{proposition}(\cite{Loday_Pir})
The universal enveloping algebra $\UU_{\Leib}(\frg)$ of a Leibniz algebra $\frg$ is the algebra generated by two copies of $\frg$ (whose elements are denoted by $r_x$ and $l_x$ respectively\footnote{$r_x$ is multiplication by $x$ from the right given by $[\ttt,x]$, respectively $l_x:= [x,\ttt]$}) subject two the following list of
 quadratic-linear identities for all pairs $x,y\in \frg$:
\begin{gather}
\label{eq::Leib::1}
r_{[x,y]} = -r_{[y,x]} = r_x r_y - r_y r_x; \\
\label{eq::Leib::2}
l_{[x,y]} = l_x r_y - r_y l_x; \\
\label{eq::Leib::3}
(r_y+l_y)l_x = 0
\end{gather}	
\end{proposition}
\begin{lemma}
The functor $\UU_{\Leib}$ does not satisfy PBW property.
However, the algebra $\UU_{\Leib}(\frg)$ admits a PBW basis and, therefore, is a nonhomogeneous Koszul algebra generated by $L(\frg):=\frg$ spanned by $l_x, x\in\frg$ and by $R(\frg):=\frg/([x,x])$ spanned by nontrivial right multiplications.  Moreover, there exists a filtration such that associated graded algebra $\gr\UU_{\Leib}(\frg)$ is isomorphic to
\begin{equation}   S(L(\frg))\otimes R(\frg).
\label{eq::ULeib::Koszul}
\end{equation}
\end{lemma}
\begin{proof}
The existence of a filtration on $\UU_{\Leib}(\frg)$  such that the associated graded is isomorphic to $S(L(\frg))\otimes R(\frg)$ was shown in~\cite{Loday_Pir}. The latter algebra is obviously Koszul.
Therefore the algebra $\UU_{\Leib}(\frg)$ is a nonhomogeneous Koszul algebra.

 However, suppose that $\frg$ is a Leibniz algebra  that has a nontrivial supercommutative part. In other words, there exists $x,y\in\frg$ such $[x,y]\neq -[y,x]$. Denote by $\frg_0$ the Leibniz algebra isomorphic to $\frg$ as a vector space but with all brackets to be zero. Then the universal enveloping algebras  $\UU_{\Leib}(\frg)$ and $\UU_{\Leib}(\frg_0)$ are generated by different spaces and, consequently, have different size, because 
\( R(\frg_0)\simeq \frg \neq R(\frg). \) Therefore, the functor $\UU_{\Leib}$ does not satisfy PBW in the sence of Definition~\ref{def::PBW1}.
\end{proof}

\subsection{Zinbiel algebras}
\label{ex::Zinb}

The operad of Zinbiel algebras, Koszul-dual to the operad $\Leib$, is generated by a binary nonsymmetric operation denoted by $(x\cdot y)$ subject to the following relation:
\[
(x\cdot y)\cdot z = x\cdot (y\cdot z) + x\cdot (z\cdot y).
\] 
\begin{proposition}
The universal enveloping functor $\UU_{\Zinb}$ satisfies PBW and the corresponding Schur functor $V\mapsto \UU_{\Zinb}(V_0)$ is isomorphic to $\kk\oplus V\oplus V \oplus V\otimes V$.
\end{proposition}
\begin{proof}
	Thanks to Proposition~\ref{prp::Leib::Grob} we know that the operad $\Zinb$ admits a quadratic Gr\"obner basis such that the set of leading monomials is the dual set to the leading monomials of Gr\"obner basis in $\Leib$. This set consists of left comb monomials and thanks to Theorem~\ref{thm::U:PBW} we get the PBW property of $\UU_{\Zinb}$. 
	
	The $\bbS_n$-representation $\Zinb(n)$ is isomorphic to the regular representation $\kk[\bbS_n]$. In particular, the symmetric collection $\cup_{n}\Zinb(n)$ is isomorphic to symmetric collection $\cup_{n}\Ass(n)$. Therefore, thanks to Theorem~\ref{thm:Hilb::ser::PBW} the corresponding Schur functors assigned to universal enveloping functor should be the same for $\Zinb$ and for $\Ass$. The $\bbS$-character of $\UU_{\Ass}$ was already discussed in~\ref{eq::UAss}.
\end{proof}

\begin{remark}
The universal enveloping functor of a Zinbiel algebra was discussed in detail in~\cite{Balavoine}.
\end{remark}

\subsection{Dendriform algebras}
The operad of Dendriform algebras was invented by J.-L.\,Loday in~\cite{Loday_Zinb}. It is shown in~\cite{Dotsenko::Free} that this operad admits a quadratic Gr\"obner basis whose leading monomials are given by certain left-most combs. Therefore the corresponding universal enveloping functor satisfies PBW.

\end{document}